\documentclass[a4paper,10pt]{amsart}
\usepackage{amsmath}
\usepackage{mathrsfs}
\usepackage{amsfonts}
\usepackage{graphicx}
\usepackage{color}
\usepackage{amsfonts}
\usepackage{amssymb}
\usepackage[hidelinks]{hyperref}

\usepackage{palatino}

\usepackage[abbrev,backrefs]{amsrefs}

\usepackage{mathtools}
\usepackage{float}

\newtheorem{theorem}{Theorem}[section]

\newtheorem{definition}[theorem]{Definition}
\newtheorem{example}[theorem]{Example}

\newtheorem{lemma}[theorem]{Lemma}

\newtheorem{proposition}[theorem]{Proposition}
\newtheorem{remark}[theorem]{Remark}

\newtheorem*{assumption}{Packing Assumption}

\newtheorem{theorem1}{Theorem}

\newtheorem{thm}{Theorem}

\newtheorem{thma}{Theorem}

\newtheorem{thmprime}{Theorem}

\numberwithin{equation}{section}

\newcommand{\interior}[1]{%
	{\kern0pt#1}^{\mathrm{o}}%
}

\def\C{{\mathcal C}}
\def\D{{\mathcal D}}

\newcommand{\rn}{\mathbb{R}^n}

\newcommand{\cont}{\mathcal{H}^\varphi_\infty}
\newcommand{\contC}{\mathcal{H}^\C_\infty}
\newcommand{\contH}{H^\lambda}
\newcommand{\dcontH}{\, {\rm d}H^\lambda}
\newcommand{\dcontC}{\, {\rm d}\mathcal{H}^\C_\infty}
\newcommand{\dc}{\,{\rm d}\C}
\newcommand{\dv}{\,{\rm d}\mathcal{H}^\C_Q}
\newcommand{\dvv}{\,{\rm d}\mathcal{H}^\varphi_\infty}
\newcommand{\Dq}{\D(Q)}

\newcommand{\dt}{\,{\rm d}t}

\newcommand{\barint}{
	\rule[.036in]{.12in}{.009in}\kern-.16in \displaystyle\int }

\newcommand{\barcal}{\mbox{$ \rule[.036in]{.11in}{.007in}\kern-.128in\int $}}

\DeclareMathOperator*{\argmin}{arg\,min}

\usepackage{enumitem}
\makeatletter
\let\@wraptoccontribs\wraptoccontribs
\makeatother

\mathchardef\mhyphen="2D

\title{The Capacitary John-Nirenberg Inequality Revisited}

\author[R. Basak]{Riju Basak}
\address[R. Basak]{Department of Mathematics, National Taiwan Normal University, No. 88, Section 4, Tingzhou Road, Wenshan District, Taipei City, Taiwan 116, R.O.C.
}
\email{rijubasak52@ntnu.edu.tw}

\author[Y.~W.~B. Chen]{You-Wei Benson Chen}
\address[Y.~W.~B. Chen]{Department of Mathematics, National Taiwan University, Taipei 10617, R.O.C.
}
\email{bensonchen.sc07@nycu.edu.tw}

\author[P. Roychowdhury]{Prasun Roychowdhury}
\address[P. Roychowdhury]{Mathematics Division, National Center for Theoretical Sciences, NTU, Cosmology Building, No. 1, Sec. 4, Roosevelt Rd., Taipei City, Taiwan 106, R.O.C.\\
	\newline
	and
	\newline
	Dipartimento di Matematica e Applicazioni, Università degli Studi di Milano–Bicocca, Via Cozzi 55, 20125 Milano, Italy
} 
\email{prasun.roychowdhury@unimib.it}

\author[D. Spector]{Daniel Spector}
\address[D. Spector]{Department of Mathematics, National Taiwan Normal University, No. 88, Section 4, Tingzhou Road, Wenshan District, Taipei City, Taiwan 116, R.O.C.\\
	\newline
	and
	\newline
	National Center for Theoretical Sciences\\No. 1 Sec. 4 Roosevelt Rd., National Taiwan
	University\\Taipei, 106, Taiwan
		\newline
	and
\newline
Department of Mathematics, University of Pittsburgh, Pittsburgh, PA 15261 USA
}
\email{spectda@gapps.ntnu.edu.tw}

\date{\today}

\begin{document}
	
	\maketitle
	
	\begin{abstract}
		In this paper, we establish maximal function estimates, Lebesgue differentiation theory, Calder\'on-Zygmund decompositions, and John-Nirenberg inequalities for translation invariant Hausdorff contents.  We further identify a key structural component of these results -- a packing condition satisfied by these Hausdorff contents which compensates for the non-linearity of the capacitary integrals.  We prove that for any outer capacity, this packing condition is satisfied if and only if the capacity is equivalent to its induced Hausdorff content.   Finally, we use this equivalence to extend the preceding theory to general outer capacities which are assumed to satisfy this packing condition. 
	\end{abstract}

	\section{Introduction}
	\subsection{Main Results}
	
	Let $\mathcal{P}(\rn)$ denote the set of all subsets of $\rn$ and suppose $\C:\mathcal{P}(\rn)\rightarrow [0,\infty]$ is an outer capacity in the sense of N. Meyers (or outer capacity for brevity), e.g. $\C$ satisfies
	\begin{itemize}
		\item[(i)] (Null set) $\C(\emptyset)=0$; 
		\item[(ii)] (Monotonicity) If $A\subset B$, then $\C(A)\leq \C(B)$;
		\item[(iii)] (Countable subadditivity) If $A\subset \cup_{i=1}^\infty A_i$, then $\C(A)\leq \sum_{i=1}^\infty\C(A_i)$.
		\item[(iv)] (Outer regularity) $\C(A) = \inf_{A \subset U, U \text{open}} \C(U)$.
	\end{itemize}
	For $f: \mathbb{R}^n \to \mathbb{R}$, further define the maximal function
	\begin{align}\label{capacitary-maximal}
		\mathcal{M}_{\C} f(x):= \sup_{r>0} \frac{1}{\C(B(x,r))}\int_{B(x,r)} |f(y)| \dc,
	\end{align}
	where the integral on the right-hand-side is intended in the sense of Choquet (see equation \eqref{choq-int} in Section \ref{sec:preliminaries} below) and $B(x,r)$ is the open ball centered at $x$ and with radius $r>0$.
	

	When $\C$ is a doubling Radon measure, estimates for the maximal function \eqref{capacitary-maximal} and the resulting Lebesgue differentiation theorems, Calder\'on-Zygmund decomposition, and John-Nirenberg inequality are well-established and classical, see e.g. \cite[p.~13-17]{Stein} and \cite[p.~14]{ABKY}.  Motivated by work on the study of critical Sobolev embeddings \cite{MS}, in the recent papers \cite{ChenSpector, ChenOoiSpector} an analogous theory has been established for $\C =\mathcal{H}^\beta_\infty$, the spherical Hausdorff content of dimension $\beta \in (0,n]$.  These results represent one facet of a broader renewed interest in the $\beta$-dimensional Hausdorff content -- and of capacities in general -- in the last decade, where a number of results have been developed \cite{CKK,FKR,HH,HH1, HH2, HKST,OP,OP1,KK,PonceSpector1,ST, STW1, STW2} which build on the classical results for such structures \cite{Adams, Adams1, AdamsPolking, AdamsChoquet,AdamsChoquet1prime,AdamsChoquet1,Anger, Choquet,Mazya,MazyaHavin,Meyers,OV}.
	
	The first aspect of this paper is the extension of the analysis for $\mathcal{H}^\beta_\infty$ to a general class of Hausdorff contents with suitable structural assumptions.  To state these results let us introduce some notations.  Throughout the paper we denote by $Q$ a half-open cube of $\rn$ with side length $\ell$,	
	\begin{align*}
		Q=[x_1,x_1+\ell)\times\cdots\times [x_n,x_n+\ell), \quad \text{ where }\, (x_1,\cdots,x_n)\in\rn, \quad\ell>0,
	\end{align*}
	and let $\D_0(Q)$ be the dyadic cubes associated with the above $Q$ defined by
	\begin{align*}
		\D_0(Q)=\{2^k(Q+m)\,:\, k\in\mathbb{Z}, \, m\in \mathbb{Z}^n\}.
	\end{align*}
	We define the extended version of the above dyadic cubes by
	\begin{align*}
		{\D}(Q)=\D_0(Q)\bigcup \left\{\prod_{i=1}^{n}I_i\,:\, I_i=[x_i,\infty) \, \text{ or } \, (-\infty,x_i)\right\}.
	\end{align*}
	We next introduce a class of translation invariant Hausdorff contents which generalize $\mathcal{H}^\beta_\infty$.  In particular, let $\varphi : [0,\infty] \to [0,\infty]$ be a monotone increasing function such that $\lim_{t \to 0^+} \varphi(t)=0$, and define the dyadic Hausdorff content associated to $\varphi$:
	\begin{align}\label{dyadiccontent}
		\cont(E):= \inf \left\{ \sum_{i} \varphi(l(Q_i)): E\subset \bigcup_{i} Q_i, Q_i\in \D(Q) \right\}.
	\end{align}

	When $\varphi(t) = t^\beta$ for $\beta \in (0,n]$, \eqref{dyadiccontent} defines the dyadic Hausdorff content of dimension $\beta$, a set function which is equivalent to the spherical content $\mathcal{H}^\beta_\infty$.  As this equivalence depends only on the underlying structure of balls and cubes in $\mathbb{R}^n$ and the monotonicity of $\varphi$, one has a similar equivalence for \eqref{dyadiccontent} and its spherical analog.  The former is an outer measure, i.e., satisfies $(i)$, $(ii)$, and $(iii)$ in the definition of an outer capacity above, while the latter is an outer capacity. As various arguments are simplified when working with dyadic cubes, because of this equivalence we prefer to introduce and work with the dyadic variants of these set functions.  The outer regularity assumption shall return at a later point for general capacities.  In the meantime, let us note that the definition of the maximal function \eqref{capacitary-maximal} makes sense for any monotone increasing set function.

	We are now prepared to state our results for this class of Hausdorff contents.  First, we have the following theorem which gives bounds for the Hausdorff content maximal function.
	\begin{thmprime}\label{weak-type-cont-maximal}
		Suppose $\varphi : [0,\infty] \to [0,\infty]$ is a monotone increasing function such that $\lim_{t \to 0^+} \varphi(t)=0$,  let $\cont$ be the dyadic content associated to $\varphi$ given in \eqref{dyadiccontent}, and let $f:\mathbb{R}^n \to \mathbb{R}$.	
		\begin{itemize}
			\item[(i)]  Then for every $t>0$, there exists a constant $C>0$ such that 
			\begin{align*}
				\cont \left(\left\{ x\in \rn \, :\, \mathcal{M}_{\cont} f(x) >t\right\}\right) \leq \frac{C}{t} \int_{\mathbb{R}^n} |f|\; \dvv.
			\end{align*}

			\item[(ii)] If $1<p<\infty$, then there exists a constant $C^{}>0$ such that 
			\begin{align*}
				\int_{\rn} (\mathcal{M}_{\cont} f)^p  \;\dvv \leq C \int_{\mathbb{R}^n} |f|^p \;\dvv.
			\end{align*}
		\end{itemize}
	\end{thmprime}	
	\noindent
	Here $\mathcal{M}_{\cont}$ denotes the Hausdorff content maximal function, defined by equation \eqref{capacitary-maximal} with $\C = \cont$.  
	
	Second, from Theorem \ref{weak-type-cont-maximal}, following the usual argument we immediately obtain a Lebesgue differentiation theorem for $L^1(\mathbb{R}^n;\cont)$, the closure of the space of continuous bounded functions with respect to the norm defined by the content.
	\begin{thmprime}\label{Lebesgue-diff-cont}
		Suppose $\varphi : [0,\infty] \to [0,\infty]$ is a monotone increasing function such that $\lim_{t \to 0^+} \varphi(t)=0$, 
		and let $\cont$ be the dyadic content associated to $\varphi$ given in \eqref{dyadiccontent}.  For $f \in L^1(\mathbb{R}^n;\cont)$, one has 
		\begin{align*}
			\lim_{r \to 0} \frac{1}{\cont(B(x,r))} \int_{B(x,r)} |f- f(x)| \dvv = 0,
		\end{align*} 
		for $\cont$ -quasi every $x\in \mathbb{R}^n$. 
	\end{thmprime}
	\noindent
	Here and in the sequel, statements which are quantified for $\C$-quasi every $x\in \mathbb{R}^n$ mean that there exists a set $E \subset \mathbb{R}^n$ such that $\C(E)=0$ and the assertion holds in $\mathbb{R}^n \setminus E$.
	
	Third, we have the following Calder\'on-Zygmund decomposition for the general Hausdorff content $\cont$.  
	\begin{thmprime}\label{cz}
		Suppose $\varphi : [0,\infty] \to [0,\infty]$ be a monotone increasing function such that $\lim_{t \to 0^+} \varphi(t)=0$, 
		and let $\cont$ be the dyadic content associated to $\varphi$ given in \eqref{dyadiccontent}.  There exists a constant $M_{0}>0$ such that for every $Q^\prime \in \Dq$, every $f \in L^1(Q^\prime;\cont)$, and $\Lambda>0$ which satisfy
		\begin{align*}
			\Lambda \geq \frac{1}{\cont(Q^\prime)} \int_{Q^\prime} |f| \dvv,
		\end{align*}
		then there exists a countable collection of non-overlapping dyadic cubes $\{ Q_k\}\subset Q^\prime$ subordinate to $Q$ such that 
		\begin{enumerate}
			\item $\Lambda< \frac{1}{\cont(Q_k)} \int_{Q_k} |f| \dvv \leq M_{0} \Lambda$,
			\item $|f(x)| \leq \Lambda$ for $\cont$ -quasi everywhere $x\in Q^\prime \setminus \bigcup_k Q_k$.
		\end{enumerate}		
	\end{thmprime}
	
	Our final result for this class of Hausdorff contents is the John-Nirenberg inequality established in our next theorem.
	\begin{thmprime}\label{jn_content}
		Suppose $\varphi : [0,\infty] \to [0,\infty]$ is a monotone increasing function such that $\lim_{t \to 0^+} \varphi(t)=0$, 
		and let $\cont$ be the dyadic content associated to $\varphi$ given in \eqref{dyadiccontent}.  Let $Q_0\in\Dq$. Then there exist constants $c,C>0$ such that
		\begin{align}\label{jninequality_content_prime}
			\cont\left(\{x\in Q^\prime:
			|f(x)-c_{Q^\prime}|>t\}\right) \leq C \cont(Q^\prime) \exp(-ct/\|f\|_{BMO^{\cont}(Q_0)}),
		\end{align}
		for every $t>0$, $f \in BMO^{\cont}(Q_0)$, all finite subcubes $Q^\prime\in \D(Q_0)$, and where 
			$$c_{Q^\prime}= \argmin_{c \in \mathbb{R}} \frac{1}{\cont(Q^\prime)}  \int_{Q^\prime}  |f-c| \dvv.$$
	\end{thmprime}
	\noindent
	Here $f \in BMO^{\cont}(Q_0)$ means that
	\begin{align*}
		\|f\|_{BMO^{\cont}(Q_0)} := \|\mathcal{M}^{\#}_{\cont} f \|_{L^\infty(Q_0)}<+\infty
	\end{align*}
	for
	\begin{align}\label{cap-sharp-maximal}
		\mathcal{M}^{\#}_{\cont} f(x):= \sup_{x \in Q' \in \D(Q_0)} \inf_{c \in \mathbb{R}} \frac{1}{\cont(Q')}\int_{Q'} |f(y)-c|\dvv
	\end{align}
	the $\cont$-sharp maximal function (which was implicitly introduced in \cite{ChenSpector} and further studied in \cite{ChenClaros} in the case $\varphi(t)=t^\beta$).	
	
	For the spherical Hausdorff contents $\mathcal{H}^\beta_\infty$, the conclusions of Theorems \ref{weak-type-cont-maximal}, \ref{Lebesgue-diff-cont}, \ref{cz}, and \ref{jn_content} are known from the results established in \cite{ChenSpector, ChenOoiSpector}.
	Given that the results hold for the $\beta$-dimensional Hausdorff content, it is not so surprising that they can be shown to hold for the class of translation invariant contents defined in the same spirit.  One may wonder, however, whether the theory in general, or at least some aspects of it, can be developed for general outer capacities.  An examination of the proofs of the preceding results suggest that an important ingredient is a packing condition enjoyed by the Hausdorff contents defined via \eqref{dyadiccontent}.  This packing condition implies quasi-additivity of the integral for certain families of cubes, which serves as a substitute for linearity in the nonlinear capacitary integrals.  This motivates us to introduce this packing condition as follows:
	\begin{assumption}[P]\label{int-cap}
		For an outer measure $\C$, we say that $\C$ satisfies the packing assumption if there exists a constant $A_{0}\geq 1$ such that for every $Q^\prime\in \Dq$ and every non-overlapping collection~$\left\{Q_{j}: Q_{j} \subset Q^\prime\right\}$ which satisfies  
		\begin{equation*}
			\sum_{Q_{j} \subset Q^\prime} \C\left(Q_{j}\right) \leq A_{0} \, \C(Q^\prime),
		\end{equation*}	
		one has
		\begin{equation*}
			\sum_{j} \int_{Q_{j}} f \dc \leq A_0\int_{\cup_{j} Q_{j}} f \dc
		\end{equation*}
		for every function $f:\rn \rightarrow [0,\infty]$.	\end{assumption}

	With this packing assumption, we establish the following theorem which asserts bounds for a dyadic analogue of the capacitary maximal function.  In particular, the following results are true not just for outer capacities, but for general outer measures -- set functions which satisfy $(i), (ii)$, and $(iii)$ in the definition of outer capacity.
	\begin{thma}\label{max-new} 
		Suppose $\C$ is an outer measure which satisfies (P) and let $f:\mathbb{R}^n \to \mathbb{R}$.		\begin{itemize}
			\item[(i)] Then for every $t>0$, there exists a constant $C>0$ such that 
			\begin{align*}
				\C \left(\left\{ x\in \rn \, :\, M^d_{\C} f(x) >t\right\}\right) \leq \frac{C}{t} \int_{\mathbb{R}^n} |f|\dc.
			\end{align*}

			\item[(ii)] If $1<p<\infty$, then there exists a constant $C>0$ such that 
			\begin{align*}
				\int_{\rn} (M^d_{\C} f)^p  \dc \leq C \int_{\mathbb{R}^n} |f|^p \dc.
			\end{align*}
		\end{itemize}
	\end{thma}
	\noindent
	Here $ M^d_{\C}f$ is the dyadic maximal function associated with $\C$ subordinate to $Q$, i.e., 
	\begin{align}\label{max-dyadic}
		M^d_{\C}f(x) : = \sup_{Q' \in \mathcal{D}(Q)} \chi_{Q' } (x) \frac{1}{\C(Q')} \int_{Q'} |f| \dc.
	\end{align}

	From this one deduces a corresponding differentiation theorem:
	
	\begin{thma}\label{Lebesgue-differentiation-new-dyadic}
		Suppose $\C$ is an outer measure and satisfies (P). For $f \in L^1(\mathbb{R}^n;\C)$, one has 
		\begin{align}\label{diff-Thm-dyadic}
			\lim_{Q' \to x} \frac{1}{\C(Q')} \int_{Q'} |f- f(x)| \dc = 0,
		\end{align}
		for $\C$ -quasi every $x\in \mathbb{R}^n$.  Here, the limit is taken over the dyadic tower of cubes $Q'\in \Dq$ containing $x$ that shrink to $x$. 	
	\end{thma}
	
	As in the case of Radon measures, it is natural then to introduce the notion of doubling for a capacity and ask whether one has analogues of Theorems \ref{weak-type-cont-maximal}, \ref{Lebesgue-diff-cont}, \ref{cz}, and \ref{jn_content}.  To this end, we say that an outer measure or outer capacity is doubling if there exists a universal constant $D_0>0$ such that
	\begin{align}\label{global_doubling}
		\C(B(x,2r)) \leq D_0 \C(B(x,r)).
	\end{align}
	We will show that one does indeed have such a theory, whenever one assumes the capacity is doubling and satisfies the packing assumption (P).  This theory depends on a somewhat surprising result concerning the seemingly unassuming packing assumption (P).  To introduce this result, for an outer capacity $\C$, we define the Hausdorff content induced by $\C$ via the formula
	\begin{align}\label{capacity_dyadiccontent}
		\contC(E):= \inf \left\{ \sum_{i} \C(Q_i): E\subset \bigcup_{i} Q_i, Q_i\in \D(Q) \right\}.
	\end{align}
	Then our next main result is the following theorem, which characterizes outer capacities in terms of their induced contents.

	\begin{theorem1}\label{com-cap}
		For an outer capacity $\C$, the following are equivalent.
		\begin{itemize}
			\item[(i)] $\C$ satisfies (P).
			\item[(ii)]  For all $E\subset\rn$ one has
			\begin{align*}
				\frac{1}{4} \mathcal{H}^\C_\infty(E)\leq \C(E)\leq \mathcal{H}^\C_\infty(E).
			\end{align*}	
		\end{itemize}
	\end{theorem1}
	Theorem \ref{com-cap} asserts that in the class of outer capacities, up to equivalence as set functions, the induced Hausdorff contents \eqref{capacity_dyadiccontent} are the only set functions which satisfy the packing assumption.  As we verify below in Section \ref{sec:preliminaries}, $\contC$ possesses many of the same structural properties as the Hausdorff contents defined via \eqref{dyadiccontent}, and therefore Theorem \ref{com-cap} and some additional analysis in the spirit of Theorems \ref{weak-type-cont-maximal}, \ref{Lebesgue-diff-cont}, \ref{cz}, and \ref{jn_content} allow us to extend these results to the class of outer capacities which satisfy our packing assumption (P) and whose induced content is doubling.  Note that while doubling did not appear explicitly in the statements of the preceding theorems, it is implicit because of the translation invariance of the contents \eqref{dyadiccontent}, see Lemma \ref{content_doubling_lemma} below in Section \ref{Max-Diff-CZ}.
	
	Precisely, we have the following results for general outer capacities.  Our first result is the following capacitary maximal theorem. 
	\begin{thm}\label{maxi-esti-new}
		Suppose $\C$ is an outer capacity which satisfies (P) and \eqref{global_doubling}, and let $f:\mathbb{R}^n \to \mathbb{R}$.
		\begin{itemize}
			\item[(i)] Then for every $t>0$, there exists a constant $C>0$ such that 
			\begin{align*}
				\C \left(\left\{ x\in \rn \, :\, \mathcal{M}_{\C} f(x) >t\right\}\right) \leq \frac{C}{t} \int_{\mathbb{R}^n} |f|\dc.
			\end{align*}
			
			\item[(ii)] If $1<p<\infty$, then there exists a constant $C>0$ such that 
			\begin{align*}
				\int_{\rn} (\mathcal{M}_{\C} f)^p  \dc \leq C \int_{\mathbb{R}^n} |f|^p \dc.
			\end{align*}
		\end{itemize}		
	\end{thm}
	\noindent
	
	Our second result is a Lebesgue differentiation theory for such capacities.
	\begin{thm}\label{Lebesgue-differentiation-new}
		Suppose $\C$ is an outer capacity which satisfies (P) and \eqref{global_doubling}.  For $f \in L^1(\mathbb{R}^n;\C)$ one has		
		\begin{align}\label{diffThm}
			\lim_{r \to 0} \frac{1}{\C(B(x,r))} \int_{B(x,r)} |f- f(x)| \dc = 0,
		\end{align} 
		for $\C$ -quasi every $x\in \mathbb{R}^n$.
	\end{thm}
	Our third result is the following Calder\'on-Zygmund decomposition.  
	\begin{thm}\label{cz_C}
		Suppose $\C$ is an outer capacity that satisfies (P) and \eqref{global_doubling}.    There exists a constant $D>0$ such that for every $Q^\prime \in \Dq$, $f \in L^1(Q^\prime;\C)$, and $\Lambda>0$ which satisfy
		\begin{align*}
			\Lambda \geq \frac{1}{\C(Q^\prime)} \int_{Q^\prime} |f| \dc,
		\end{align*}
		then there exists a countable collection of non-overlapping dyadic cubes $\{ Q_k\}\subset Q^\prime$ subordinate to $Q$ such that 
		\begin{enumerate}
			\item $\Lambda< \frac{1}{\C(Q_k)} \int_{Q_k} |f| \dc \leq D \Lambda$,
			\item $|f(x)| \leq \Lambda$ for $\C$ -quasi everywhere $x\in Q^\prime \setminus \bigcup_k Q_k$.
		\end{enumerate}
	\end{thm}
	\begin{remark}\label{doubling_consequence}
		The constant $D$ appearing in Theorem \ref{cz_C} is the constant for which 
		\begin{align}\label{global_doubling_dyadic}
			\C(P') \leq D \C(Q'),
		\end{align}
		for any $Q' \in \mathcal{D}(Q)$, where $P' \in \mathcal{D}(Q)$ is the parent of $Q'$.  That $\C$ satisfies such a condition follows from repeated application of the assumption \eqref{global_doubling} and monotonicity, as a ball contained in $Q'$ can be dilated to contain $P'$, the amount of dilation depending on the dimension.
	\end{remark}
	
	Finally, we have a capacitary John-Nirenberg inequality.
	\begin{thm}\label{jn_content_C}
		Suppose $\C$ is an outer capacity that satisfies (P) and \eqref{global_doubling}. Let $Q_0\in\Dq$.
		Then there exist constants $a,A>0$ such that
		\begin{align*}
			\C\left(\{x\in Q^\prime:|f(x)-c_{Q^\prime}|>t\}\right) \leq A \C(Q^\prime) \exp(-at/\|f\|_{BMO^{\C}(Q_0)})
		\end{align*}
		for every $t>0$, $f \in BMO^{\C}(Q_0)$, all finite subcubes $Q^\prime\in \D(Q_0)$, and where $$c_{Q^\prime}= \argmin_{c \in \mathbb{R}} \frac{1}{\C(Q^\prime)}  \int_{Q^\prime}  |f-c| \dc.$$
	\end{thm}
	Here $f \in BMO^{\C}(Q_0)$ means that
	\begin{align*}
		\|f\|_{BMO^{\C}(Q_0)} := \|\mathcal{M}^{\#}_{\C} f \|_{L^\infty(Q_0)}<+\infty
	\end{align*}
	for
	\begin{align*}
		\mathcal{M}^{\#}_{\C} f(x):= \sup_{x \in Q \in \D(Q_0)} \inf_{c \in \mathbb{R}} \frac{1}{\mathcal{C}(Q')}\int_{Q'} |f(y)-c|\dc
	\end{align*}
	the capacitary sharp maximal function.


	The plan of the paper is as follows.
	In Section \ref{examples}, we present some examples of the theory discussed in the introduction.  
	In Section \ref{sec:preliminaries}, we recall some background regarding general Hausdorff contents induced by monotone increasing set functions as well as establish several useful results which were previously known for the $\beta$ dimensional Hausdorff content. In Section \ref{sec:Dyadic Analysis for Outer Measures}, we give the proofs of Theorems \ref{max-new} and \ref{Lebesgue-differentiation-new-dyadic} for general outer measures. In Section \ref{Max-Diff-CZ}, we establish results related to a translation invariant Hausdorff contents $\cont$. In particular, we prove weak type $(1,1)$ and strong type $L^p$ estimate for maximal function in Theorem \ref{weak-type-cont-maximal}, Lebesgue differentiation theorem in Theorem \ref{Lebesgue-diff-cont}, Calder\'on-Zygmund decomposition in Theorem \ref{cz}. In Section \ref{JN-Inequality}, we prove John-Nirenberg inequality for $\cont$ in Theorem \ref{jn_content}. Finally, in Section \ref{capacities}, we prove Theorems \ref{maxi-esti-new}, \ref{Lebesgue-differentiation-new}, \ref{cz_C} and \ref{jn_content_C} which includes Maximal function estimates, Lebesgue differentiation theorem, Calder\'on-Zygmund decomposition and John-Nirenberg inequality for general doubling outer capacity.

	\subsection{Examples}\label{examples}
	\begin{example}\label{eg-1}
		Let $\beta\in (0,n]$ and define the spherical Hausdorff content of dimension $\beta$ by
		\begin{align*}
			\mathcal{H}^\beta_\infty(E):= \inf \left\{ \sum_{i} \omega_{\beta} r_i^\beta: E\subset \bigcup_{i} B(x_i,r_i)\right\},
		\end{align*}
		where $\omega_{\beta}= \pi^{\beta/2}/\Gamma(\beta/2+1)$ is a normalization constant.  For Theorems \ref{weak-type-cont-maximal}, \ref{Lebesgue-diff-cont}, \ref{cz}, and \ref{jn_content} were established for $\mathcal{H}^\beta_\infty$ in \cite{ChenOoiSpector,ChenSpector}.  Of course, they also follow from the more general arguments of the present paper, either by the equivalence of  $\mathcal{H}^\beta_\infty$ and its dyadic analogue (see e.g. Proposition 2.3 in \cite{Yang-Yuan}) and making a direct appeal to Theorems \ref{weak-type-cont-maximal}, \ref{Lebesgue-diff-cont}, \ref{cz}, and \ref{jn_content} or from the observation that $\C=\mathcal{H}^\beta_\infty$ is an outer capacity which satisfies (P) and \eqref{global_doubling} and an appeal to Theorems \ref{maxi-esti-new}, \ref{Lebesgue-differentiation-new}, \ref{cz_C}, and \ref{jn_content_C}.
	\end{example}

	\begin{example}
		Fix $1\leq p<n$, set $p^*=\frac{np}{n-p}$, and define
		\begin{align*}
			K^p:=\{f:\rn\rightarrow\mathbb{R}\,|\,f\geq 0, \, f\in L^{p^*}(\rn), \, \nabla f \in L^p(\rn,\rn)\}.
		\end{align*}
		Now, for $E\subset \rn$, we define Sobolev capacity by
		\begin{align*}
			\text{Cap}_{1,p}(E):=\inf \left\{\int_{\rn}|\nabla f|^p\,{\rm d}x\,|\, f\in K^p, E\subset \interior{\{f\geq 1\}}\right\},
		\end{align*}
		where $\interior{\{f\geq 1\}}$ denotes the interior of the set $\{f\geq 1\}$. Theorem 4.15 in \cite{EvansGariepy} shows that $\C=\text{Cap}_{1,p}$ is a strongly subadditive (see Section \ref{sec:preliminaries} below) outer capacity and satisfies \eqref{global_doubling}.  This capacity does not satisfy (P), as this would imply $\C \sim \mathcal{H}^\C_\infty$, which is known to be false (this follows from a suitable modification of the Cantor set construction in \cite[Section 5.3]{AdamsHedberg}, see also the next example).
	\end{example}
	
	The next example generalizes the preceding to a broader range of exponents.
	
	\begin{example}
		Let $\alpha\in (0,n)$ and $1<p<\frac{n}{\alpha}$.  Define the Riesz capacity of a set $E \subset \mathbb{R}^n$ by
		\begin{align*}
			\operatorname*{cap_{\alpha,p}}(E)	= \inf\left\{||\psi||^p_{L^p(\rn)}\,:\, \psi \geq 0, \, I_\alpha *\psi\geq 1 \, \text{ on }\, E\right\}, \quad E\subset\rn,
		\end{align*}
		where
		\begin{align*}
			I_\alpha(x): =\frac{1}{\gamma(\alpha)}\frac{1}{|x|^{n-\alpha}}
		\end{align*}
		denotes the Riesz kernel of order $\alpha$ and $\gamma(\alpha)$ is a normalization constant given in \cite[p.~117]{S}. Then $\operatorname*{cap_{\alpha,p}}$ is an outer capacity and satisfies \eqref{global_doubling}.  Theorems \ref{maxi-esti-new} and \ref{Lebesgue-differentiation-new} were proved for $\operatorname*{cap_{\alpha,p}}$ in \cite{ChenOoiSpector}, and while Theorem \ref{cz_C} has not appeared explicitly, it follows from the argument in \cite[Theorem 3.4]{ChenSpector}.  As in the preceding example, this capacity does not satisfy (P), as this would imply $\C \sim \mathcal{H}^\C_\infty$, which can easily be shown not to hold by \cite[Theorem 5.3.2 on p.~143]{AdamsHedberg}.  We do not know whether $\operatorname*{cap_{\alpha,p}}$ supports a John-Nirenberg inequality.
	\end{example}
	
	In the preceding examples involving Sobolev and Riesz capacities, the content induced by the set function is precisely a $\beta$-dimensional Hausdorff content.  This suggests one to consider two families of examples which arise in the context of Sobolev embeddings.  
	
The first family appears in Sobolev embeddings away from the case $p=1$, precisely because of the discrepancy between the Sobolev capacity and Hausdorff content mentioned in the preceding examples.  In particular, as recorded in \cite[p.~361 in 1985]{Mazya}, trace inequalities are shown to hold for powers of measures which are smaller than a Sobolev capacity.
	\begin{example}
		Let $\mu$ be a locally finite Radon measure which is doubling, fix $\alpha \in (0,n]$, and let 
		\begin{align*}
			\C(E)=\inf_{U\supset E} \mu(U)^{\alpha/n}
		\end{align*} 
        where infimum is taken over all open sets $U$ contains $E$.
		Then $\C$ is an outer capacity and satisfies the doubling condition \eqref{global_doubling}.  Therefore, Theorems~\ref{maxi-esti-new}, \ref{Lebesgue-differentiation-new}, \ref{cz_C}, and \ref{jn_content_C} are applicable to $\contC$.  
\end{example}	
	
	The second family, which appears in embeddings in the Sobolev critical exponent, is a case where one has a logarithmic capacity.  While the capacity and content are again not equivalent, the family of contents generated this way gives more interesting examples.  

	\begin{example}\label{spl-eg}
		Let $\beta \in (0,n]$.  For $t>0$, define
		\[
		\varphi(t)= \begin{cases} \left[\log \left(\frac{2}{t}\right) \right]^{-\beta} & \text { if }\, t<2; \\ \infty & \text { if } \, t\geq 2. \end{cases}
		\]
		This choice of $\varphi$ does not lead to a doubling set function, though it does define a locally doubling set function, i.e., for $Q' \in \mathcal{D}(Q)$ with $\ell(Q')<1$, we have
		\begin{align*}
			\frac{\varphi(\ell(P^\prime)) }{\varphi(\ell(Q^\prime))}=	\left[\log \left(\frac{2}{2\ell(Q^\prime)}\right)/\log \left(\frac{2}{\ell(Q^\prime)}\right) \right]^{-\beta}\leq 2^\beta,
		\end{align*}	
		where $P^\prime$ is the parent of $Q'$ in $\mathcal{D}(Q)$.  Nonetheless, the translation invariance of $\cont$ implies that one has for this set function the conclusions of Theorems \ref{weak-type-cont-maximal}, \ref{Lebesgue-diff-cont}, \ref{cz}, and \ref{jn_content}.
	\end{example}

	\section{Preliminaries}\label{sec:preliminaries}
	Throughout the paper, we work with the Choquet integrals of various set functions, whose definitions we recall here.  For a set function $H: \mathcal{P}(\mathbb{R}^n) \to [0,\infty]$ and a function $f:\mathbb{R}^n \to [0,\infty],$ we define the {\bf Choquet integral} of $f$ with respect to $H$ by 
	\begin{align}\label{choq-int}
		\int f ~ \rm d H := \int_{0}^\infty H \left(\{f>t\}\right)\dt,
	\end{align}
	where the right-hand-side is interpreted as a Lebesgue or improper Riemann integral of the monotone increasing set function
	\begin{align*}
		t \mapsto H \left(\{f>t\}\right).
	\end{align*}
		
	Let $\lambda: \D(Q) \to [0,\infty]$ denote a monotone increasing set function on the dyadic lattice generated by $Q$.  
	We next state and prove a generalization of a packing lemma that is attributed to Melnikov in the case $\lambda(Q)=l(Q)^\beta$ in \cite{OV}.
	\begin{proposition}\label{decom-cubes}
		Let $\lambda : \D(Q) \to [0,\infty]$ be a monotone increasing set function and $\left\{Q_{j}\right\}_j\subset\Dq$ be a family of nonoverlapping dyadic cubes.  Then there exists a subfamily $\left\{Q_{j_{k}}\right\}_k$ and a family of nonoverlapping ancestors $\{\tilde{Q}_{k}\}_k$ such that:
		\begin{itemize}
			\item[(1)]  We have
			\begin{align}\label{pack-1}
				\bigcup_{j} Q_{j} \subset \bigcup_{k} Q_{j_{k}} \cup \bigcup_{k} \tilde{Q}_{k}.
			\end{align}
			\item[(2)] We have
			\begin{align}\label{pack-2}
				\sum_{Q_{j_{k}} \subset Q^\prime} \lambda(Q_{j_{k}}) \leq 2 \lambda(Q^\prime), \text { for each } Q^\prime\in\Dq.
			\end{align}
			\item[(3)] For each $\tilde{Q}_{k}$,
			\begin{align}\label{pack-3}
				\lambda(\tilde{Q}_{k}) \leq \sum_{Q_{j_{i}} \subset \tilde{Q}_{k}} \lambda(Q_{j_{i}}).
			\end{align}
		\end{itemize}
	\end{proposition}
	\begin{proof}
		If for any $Q^\prime\in \Dq$, there is no cube from $\{Q_j\}_j$ contained inside $Q^\prime$, then property (2) is automatically satisfied. Again for any $Q^\prime\in\Dq$, if there is only one member from the sequence inside $Q^\prime$, the property (2) immediately follows from the monotonicity of $\lambda$. In other cases $\left\{Q_{j_{k}}\right\}$ will be selected as follows. We select $Q_{j_{1}}$ to be $Q_{1}$ and for $j_{1}, j_{2}, \ldots, j_{k}$, we select $j_{k+1}$ to be the smallest $j_{k+1}$ such that $j_{k+1}>j_{k}$ and $\left\{Q_{j_{1}}, Q_{j_{2}}, \ldots Q_{j_{k+1}}\right\}$ satisfies (2). Then for every $Q_{m}$ such that $Q_{m} \notin\left\{Q_{j_{k}}\right\}_{k}$ , there exists $Q_{m}^{*}\in \Dq$ such that $Q_{m} \subset Q_{m}^{*}$ and
		\begin{align}\label{packing-eq-1}
			\sum_{j_{k}<m, Q_{j_{k}} \subset Q_{m}^*} \lambda(Q_{j_{k}})+\lambda(Q_{m})>2 \lambda(Q_{m}^{*})	.
		\end{align}
		Also, we are taking such $Q_m^*$, for which $\lambda(Q_m^*)<\infty$. Otherwise, when $\lambda(Q_m^*)=\infty$, property (2) will be automatically satisfied, and then for $Q_m$, $Q_m^*$ will not be the right candidate. From the monotonicity of $\lambda$, we have 
		\begin{align}\label{packing-eq-2}
			\lambda(Q_m)\leq 	\lambda(Q_m^*).
		\end{align}
		Now, we let $\{\tilde{Q}_{k}\}_k$ be the maximal family of all such $Q_{m}^{*}$. Moreover, we notice that for each $Q_m^*$, there holds
		\begin{align}\label{packing-eq-3}
			\sum_{Q_{j_{k}} \subset Q_{m}^*} \lambda(Q_{j_{k}}) \geq \sum_{j_{k}<m, Q_{j_{k}} \subset Q_{m}^*} \lambda(Q_{j_{k}}) \overset{\eqref{packing-eq-1}}{>} 2 \lambda(Q_{m}^{*})-\lambda(Q_{m}) \overset{\eqref{packing-eq-2}}{\geq} \lambda \left(Q_{m}^{*}\right).
		\end{align}
		As, $Q_m^*$'s are basically $\tilde{Q}_{k}$, and hence by \eqref{packing-eq-3}, we get the property (3). By our construction, property (1) is automatically verified.
	\end{proof}
	
	Following the notational convention of \cite{STW}, we define the Hausdorff content induced by $\lambda$ by
	\begin{align}\label{lambda_dyadiccontent}
		\contH(E):= \inf \left\{ \sum_{i} \lambda(Q_i): E\subset \bigcup_{i} Q_i, Q_i\in \D(Q) \right\}.
	\end{align} 
	Observe that this definition contains $\cont$ for the choice $\lambda(Q)= \varphi(l(Q))$ and $\contC$ for the choice $\lambda(Q) = \C(Q)$.  In particular, the results we record here for $\contH$ are valid for both $\cont$ and $\contC$.  
	
	We begin with the following result which shows that the Hausdorff contents are outer capacities.
	\begin{proposition}
		Let $\lambda : \D(Q) \to [0,\infty]$ be a monotone increasing set function which satisfy $\lim_{|Q| \to 0^{+}} \lambda(Q)=0$.  Then the set function $\contH: \mathcal{P}(\mathbb{R}^n) \to [0,\infty]$ is an outer measure, e.g.~ $\contH$ satisfies $(i),(ii), (iii)$ in the definition of an outer capacity.
	\end{proposition}
	\begin{proof}
		\noindent
		\begin{itemize}
			\item[(i)] As the empty set is contained in any tower of cubes that shrinks to a point, monotonicity of $\lambda$ and the assumption that $\lim_{|Q| \to 0^{+}} \lambda(Q)=0$ implies that $\contH(\emptyset)=0$.
			\item[(ii)]  Let $B\subset \cup_{i} Q_i$, with $Q_i\in \Dq$. Since $A\subset B$, we have $A \subseteq  \cup_i Q_i$. Now from the definition of $\contH(A)$ we have
			\begin{align*}
				\contH(A)\leq \sum_{i} \lambda(Q_i).
			\end{align*}
			Therefore, taking all possible covers of $B$ and using the infimum we deduce
			\begin{align*}
				\contH(A)\leq \contH(B).
			\end{align*} 
			\item[(iii)] We may assume that $\contH(A_i)<+\infty$ for all $i$ or there is nothing to prove.  Let $\epsilon>0$.  For each $i$, let $\{ Q^i_k \}_k \subseteq \mathcal{D}(Q)$, be a cover of $A_i$ such that
			\begin{align*}
				\sum_{k}\lambda(Q^i_k) \leq \contH(A_i) +  \frac{\epsilon}{2^k}.
			\end{align*}		
			Then $A \subset \cup_i A_i$ implies 
			\begin{align*}
				A\subset\bigcup_{i,k} Q^i_k.
			\end{align*}
			As a result, we obtain
			\begin{align*}
				\contH(A)&\leq \sum_{i}\sum_{k}\lambda(Q^i_k) \\
				&\leq \sum_{i}\contH(A_i) +  \epsilon.
			\end{align*}
			The claimed countable subadditivity then follows by sending $\epsilon \to 0$.
		\end{itemize}	
	\end{proof}

	
	We next recall that a result which shows that a general Hausdorff content constructed on a dyadic lattice is strongly subadditive, established in Proposition 3.5 in \cite{STW} under more general hypotheses.  
	\begin{proposition}\label{strong-subadditivity}
		Let $\lambda :\D(Q) \to [0,\infty]$ be a monotone increasing set function.  For any sets $A, \, B\subset\rn$,
		\begin{align*}
			\contH(A\cup B)+\contH(A\cap B)\leq \contH(A)+\contH(B).
		\end{align*}
	\end{proposition}
	
	It is known that the Choquet integral constructed from a monotone increasing set function is sublinear if and only if the set function is strongly subadditive (see, e.g. \cite[Theorem 1.2]{PonceSpector}.  This fact, in combination with Proposition \ref{strong-subadditivity}, yields that the Choquet integral with respect to $\contH$ is sublinear i.e., for any $f,\,g:\rn \rightarrow [0,\infty]$, we have
	\begin{align}\label{sub-lin}
		\int (f+g) \dcontH \leq \int f \dcontH +\int g \dcontH.
	\end{align}
	
	This sublinearity, and the definition of the Choquet integral, imply the following results which are useful for our purposes.
	\begin{proposition}\label{sublinear}
		The Choquet integral with respect to the content $\contH$ satisfies the following properties:
		\begin{enumerate}
			\item  If $c \geq 0$, $f \geq 0$, then
			\begin{align*}
				\int cf\dcontH = c \int f\dcontH;
			\end{align*}
			\item \label{sub} For non-negative functions $f_n$,  
			\begin{align*}
				\int \sum_{n=1}^\infty f_n\dcontH \leq \sum_{n=1}^\infty \int f_n\dcontH.
			\end{align*}
		\end{enumerate}
	\end{proposition}

	Let $Q_0\subset\rn$ be a cube, possibly an infinite cube.  The inequality \eqref{sub-lin} is just the triangle inequality for the positively homogeneous functional
	\begin{align}\label{choquet_def}
		\|f\|_{L^1(Q_0;\contH)}\vcentcolon= \int_{Q_0} |f|\dcontH,
	\end{align}
	which when restricted to a suitable class of functions yields a normed, complete space.  To this end, we next introduce the notion of $\contH$-quasicontinuity.
	\begin{definition}
		We say that a function $f$ is $\contH$-quasicontinuous if  for every $\epsilon>0$, there exists an open set $O_\epsilon$ such that $\contH(O_\epsilon)<\epsilon$ and $f|_{O_\epsilon^c}$ is continuous.
	\end{definition}
	We can now define a natural Banach space associated with the content.  For $Q_0\subset\rn$, we define
	\begin{align}\label{l1}
		L^1(Q_0;\contH)\vcentcolon= \left\{ f\; \text{ is } \contH\text{-quasicontinuous } : \|f\|_{L^1(Q_0;\contH)}<+\infty\right\}.
	\end{align}
	We adopt the notation $L^1(\contH) = L^1(\mathbb{R}^n;\contH) $.
	
	We have the following theorem for this function space.  
	\begin{theorem}\label{functionspace}
		Let $Q_0 \subset \mathbb{R}^n$ be a cube and $L^1(Q_0;\contH)$ be defined as \eqref{l1}.  Then
		\begin{enumerate}
			\item \eqref{choquet_def} is a norm;
			\item $L^1(Q_0;\contH)$ is a Banach space;
			\item The space of continuous and bounded functions in $Q_0$, denoted as $C_b(Q_0)$, for which \eqref{choquet_def} is finite is dense in $L^1(Q_0;\contH)$.
		\end{enumerate}
	\end{theorem}
	
	\begin{proof}
		As $\contH$ is a strongly subadditive outer measure, items (i), (ii) follow from \cite[Proposition 6.1]{PonceSpector}, which asserts the validity of such a result for any monotone, strongly subadditive, countably subadditive set functions.  Item (iii) follows from  \cite[Proposition 3.2]{PonceSpector}, where the density of bounded, continuous functions in the space of $\contH$-quasicontinuous functions with finite integral is shown for monotone, finitely subadditive set functions.
	\end{proof}
	
\begin{remark}\label{density_lp}
In the sequel, we will also have reason to work with
 \begin{align*}
		L^p(\mathbb{R}^n;H)\vcentcolon= \left\{ f\; \text{ is } H\text{-quasicontinuous } : \|f\|_{L^p(\mathbb{R}^n;H)}<+\infty\right\}, ~ 1\leq p <\infty,
	\end{align*}
for $H=$ an outer capacity $\C$ or outer measure, where quasicontinuity is defined analogously to the case of $\contH$.  The relevant fact here is that one still has a density of the space of continuous and bounded functions $C_b(\mathbb{R}^n)$ for which \eqref{choquet_def} is finite in $L^p(\mathbb{R}^n;H)$, which follows from minor modifications of the proof of  \cite[Proposition 3.2]{PonceSpector}.
\end{remark}

	Besides strong subadditivity, one useful consequence of the definition of general Hausdorff contents is that they satisfy the packing assumption, as we next show in 
	\begin{proposition}\label{deco-cube-int}
		Consider a nonoverlapping dyadic collection $\{Q_{j}\}_j$ in $\Dq$. Moreover, if there exists a constant $A_{0}\geq 1$ such that	
		\begin{equation}\label{pack-int-hyp-1}
			\sum_{Q_{j} \subset Q^\prime} \contH\left(Q_{j}\right) \leq A_{0} \, \contH(Q^\prime) 
		\end{equation}	
		holds for each $Q^\prime \in \Dq$, then 	
		\begin{equation}\label{pack-int-hyp-2}
			\sum_{j} \int_{Q_{j}} f \dcontH \leq A_0\int_{\cup_{j} Q_{j}} f \dcontH.
		\end{equation}
	\end{proposition}
	\begin{proof}
		We first claim that
		\begin{equation}\label{pack-int-eq-0}
			\sum_{j} \contH(A \cap Q_{j}) \leq A_0 \contH(A)
		\end{equation}
		holds for every $A \subseteq \cup_{j} Q_{j}$. To this end, let $\{\tilde{Q}_{k}\}_{k}$ be a collection of non overlapping cubes of $\Dq$ covering $A$ and we split the collection $\{\tilde{Q}_{k}\}_{k}$ as follows:
		\begin{align*}
			\{\tilde{Q}_{k} \}_{k}=\{\cup_{j} \mathcal{A}_{j} \}\bigcup \mathcal{A}_0,	
		\end{align*}
		where for each $j\geq1$, 
		\begin{align*}
			\mathcal{A}_{j}=\left\{\tilde{Q}_{k}: \tilde{Q}_{k} \text { is strictly contained in } Q_{j}\right\},
		\end{align*}
		and
		\begin{align*}
			\mathcal{A}_0=	\left\{\tilde{Q}_{k}: \tilde{Q}_{k} \text { contains or equals to one of } Q_{j}\right\}.
		\end{align*}

		Similarly, we also divide the collection $\{Q_{j}\}_{j}$ as follows:
		\begin{align*}
			\{{Q}_{j} \}_{j}= \{\cup_{k} \mathcal{B}_{k} \} \bigcup \mathcal{B}_0,		
		\end{align*}
		where for each $k\geq1$, 
		\begin{align*}
			\mathcal{B}_{k}=\left\{Q_{j}: Q_{j} \text { is contained in or equals to } \tilde{Q}_{k}\right\},
		\end{align*}
		and
		\begin{align*}
			\mathcal{B}_0=	\left\{Q_{j}: Q_{j} \text { strictly contains one of } \tilde{Q}_{k}\right\}.
		\end{align*}
		As $\{\tilde{Q}_{k}\}_{k}$ is a cover of $A$, then for $j \in \mathbb{N}$ satisfies $Q_{j} \in \mathcal{B}_{0}$, we have
		\begin{align*}
			\bigcup_{\tilde{Q}_{k} \in \mathcal{A}_{j}} \tilde{Q}_{k} \supset Q_{j} \cap A.	
		\end{align*}
		Then, using the countable subadditivity and the monotonicity of $\contH$, we get
		\begin{align*}
			\sum_{\tilde{Q}_{k} \in \mathcal{A}_{j}} \contH (\tilde{Q}_{k}) \geq \contH\left(\bigcup_{\tilde{Q}_{k} \in \mathcal{A}_{j}} \tilde{Q}_{k}\right) \geq \contH (Q_{j} \cap A),
		\end{align*}
		and taking both side summation over $j$ in above, we obtain
		\begin{align}\label{pack-int-eq-1}
			\sum_{j} \sum_{\tilde{Q}_{k} \in A_{j}} \contH(\tilde{Q}_{k}) \geq \sum_{Q_{j} \in \mathcal{B}_{0}} \contH(Q_{j} \cap A).
		\end{align}
		
		Moreover, by \eqref{pack-int-hyp-1}, we observe that for every $k \in \mathbb{N}$, we have
		\begin{align*}
			\contH(\tilde{Q}_{k}) \geq \frac{1}{A_0} \sum_{Q_{j} \subset \tilde{Q}_{k}} \contH (Q_{j}).
		\end{align*}
		Now taking summation over all the $k$ such that $\tilde{Q}_{k} \in \mathcal{A}_{0}$, we get
		\begin{align}\label{pack-int-eq-2}
			\sum_{\tilde{Q}_{k} \in \mathcal{A}_{0}} \contH(\tilde{Q}_{k}) & \geq \frac{1}{A_0} \sum_{k} \sum_{Q_{j} \subset \tilde{Q}_{k}} \contH (Q_{j}) \\
			& =\frac{1}{A_0} \sum_{k} \sum_{Q_{j} \in \mathcal{B}_{k}} \contH (Q_{j}) \notag \\
			& \geq \frac{1}{A_0} \sum_{k} \sum_{Q_{j} \in \mathcal{B}_{k}} \contH (Q_{j} \cap A)	\notag.
		\end{align}
		
		The combination of \eqref{pack-int-eq-1} and \eqref{pack-int-eq-2} yields that
		\begin{align*}
			\sum_{k} \contH (\tilde{Q}_{k}) \geq \frac{1}{A_0} \sum_{j} \contH (Q_{j} \cap A )	.
		\end{align*}
		By definition of $\contH$, we have $\sum_{k} \lambda (\tilde{Q}_{k})\geq \sum_{k} \contH (\tilde{Q}_{k})$. Next, taking the infimum over all the families of nonoverlapping cubes of $\Dq$, $\{\tilde{Q}_{k}\}_{k}$ covering $A$, we obtain
		\begin{align*}
			\contH(A)\geq \frac{1}{A_0} \sum_{j} \contH (Q_{j} \cap A )
		\end{align*}
		and this proves the claim \eqref{pack-int-eq-0}. Consequently, we obtain
		\begin{align*}
			\sum_{j} \int_{Q_{j}} f \dcontH & =\sum_{j} \int_{0}^{\infty} \contH(\left\{x \in Q_{j}: f(x)>t\right\}) \dt \\
			& =\int_{0}^{\infty} \sum_{j} \contH(\left\{x \in Q_{j}: f(x)>t\right\}) \dt \\
			& \leq \int_{0}^{\infty}A_0 \contH(\left\{x \in \cup_{j} Q_{j}: f(x)>t\right\})\dt \\
			& =A_0 \int_{\cup_{j} Q_{j}} f \dcontH
		\end{align*}
		which completes the proof.	
	\end{proof}

	\begin{remark}
		Because $\contH$ is a monotone increasing set function, the packing decomposition established in Proposition \ref{decom-cubes} is satisfied by $\contH$.  In particular, for $\{Q_j\}_j \subset \mathcal{D}(Q)$ a family of non-overlapping dyadic cubes, there exists a subfamily $\{Q_{j_k}\}_k$ and a family of disjoint ancestors $\{\tilde{Q}_k\}_k$ such that the following conditions hold:
		\begin{itemize}
			\item[(i)] $\bigcup_j Q_j \subset \bigcup_k Q_{j_k} \cup \bigcup_k \tilde{Q}_k$,
			\item[(ii)] $\sum_{Q_{j_k} \subset Q'} \contH(Q_{j_k}) \leq 2 \contH(Q'), \quad \text{for each} \ Q' \in \mathcal{D}(Q)$,
			\item[(iii)] For each $\tilde{Q}_k$, there holds $\contH(\tilde{Q}_k) \leq \sum_{Q_{j_i} \subset \tilde{Q}_k} \contH(Q_{j_i})$.
		\end{itemize}
	\end{remark}

	\section{Dyadic Analysis for Outer Measures}\label{sec:Dyadic Analysis for Outer Measures}
	
	In this section, we give the proofs of our results for general outer measures, Theorems \ref{max-new} and \ref{Lebesgue-differentiation-new-dyadic}.  We note that the definition of the Choquet integral with respect to an outer measure is as defined in \eqref{choq-int}.
	In the sequel, we adopt the notation 
	\begin{align*}
		a \cdot \infty = \infty \cdot a = 
		\begin{cases} 
			\infty & \text{if } 0 < a \leq \infty, \\
			0 & \text{if } a = 0.
		\end{cases}
	\end{align*}
	In particular, we have
	\begin{align}\label{Rudinnotation}
		\frac{1}{\C(Q')} \int_{Q'} |f| \, \dc = 0 \quad \text{ if }\, \C(Q') = 0.
	\end{align}

	We first prove the maximal function estimate.
	
	\begin{proof}[Proof of Theorem a]
		(i)	Without loss of generality, we suppose that $f$ is non-negative.  By the definition of the dyadic maximal function, for every $t>0$, we have 
		\begin{align*}
			\left\{ x\in \rn \, :\, M^d_{\C} f(x) >t\right\} = \bigcup_{i=1}^\infty Q_i,
		\end{align*}
		where $\{Q_i\}_i\subset\Dq$ are non overlapping and each of the $Q_i$ satisfies
		\begin{align}\label{maximal_cube_t}
			\C(Q_i) < \frac{1}{t} \int_{Q_i} f \dc.
		\end{align}
		(We note that $\C(Q_i)$ must bigger than zero by (\ref{Rudinnotation})).
		Let $\{Q_j \}_j$ be the maximal collection of these dyadic cubes, each $Q_j$ of which satisfies \eqref{maximal_cube_t} and let $\{Q_{j_k}\}$ and $\{\tilde{Q}_k\}$ be the families of cubes obtained by an application of Proposition \ref{decom-cubes} to this family.  By only selecting necessary cubes,  $Q_{j_k} \not \subseteq \tilde{Q}_m$ for some $Q_m$, in the original family with $\tilde{Q}_m$ selected as a covering cube, we have 
		\begin{align*}
			\C \left(\left\{ x\in \rn \, :\, M^d_{\C} f(x)>t\right\}\right) &\leq \sum_{k, Q_{j_k} \not \subseteq \tilde{Q}_m} \C(Q_{j_k}) +\sum_k \C(\tilde{Q}_k) \\
			&\leq \sum_{k} \C(Q_{j_k}).
		\end{align*}
		However, since $Q_{j_k}$ satisfies \eqref{maximal_cube_t} we find that
		\begin{align*}
			\C \left(\left\{ x\in \rn \, :\, M^d_{\C} f(x) >t\right\}\right)  \leq  \sum_{k} \frac{1}{t} \int_{Q_{j_k}} f \dc.
		\end{align*}
		Finally, inequality  \eqref{pack-int-hyp-2} with $A_0 =2$ implies
		\begin{align*}
			\sum_{k} \frac{1}{t} \int_{Q_{j_k}} f \dc \leq \frac{2}{t} \int_{\cup_k Q_{j_k}} f \dc \leq \frac{2}{t}\int_{\rn}f \dc,
		\end{align*}
		which gives the desired conclusion of part (i).
		
		We next give the proof of part (ii). It is easy to notice that $M^d_{\C}$ is quasi-linear as for any $x\in \rn$, we have
		\begin{align*}
			|M^d_{\C} (f+g)(x)|\leq 2 \left( |M^d_{\C} f(x)| + |M^d_{\C} g(x)|\right).
		\end{align*}
		Since, 
		\begin{align*}
			||f||_{L^\infty(\C)}:=\inf\{t>0\, :\, \C(\{y\in \rn \, :\, |f(y)|>t\})=0\}
		\end{align*}
		and therefore for any $Q^\prime\in\Dq$ and $t>||f||_{L^\infty(\C)}$, 
		\begin{align*}
			\C(\{y\in Q^\prime\,:\, |f(y)|>t\}) =0.
		\end{align*}
		So, we deduce
		\begin{align*}
			\frac{1}{\C(Q')} \int_{Q'} |f| \, d\C &	= \frac{1}{\C(Q')} \int_0^\infty \C(\{y\in Q^\prime\,:\, |f(y)|>t\})\dt\\&=\frac{1}{\C(Q')} \int_0^{||f||_{L^\infty(\C)}} \C(\{y\in Q^\prime\,:\, |f(y)|>t\})\dt\\& \leq ||f||_{L^\infty(\C)}.
		\end{align*}
		Finally, taking the supremum of all $Q^\prime\in \Dq$, we get
		\begin{align*}
			||M^d_{\C} f||_{L^\infty(\C)} \leq ||f||_{L^\infty(\C)}.
		\end{align*}
		Hence, applying \cite[Lemma 2.5]{ChenOoiSpector} for $T=M^d_{\C}$, we complete the proof of part (ii).
	\end{proof}
	

We now proceed to give the proof of Theorem \ref{Lebesgue-differentiation-new-dyadic}.
	
	\begin{proof}[Proof of Theorem \ref{Lebesgue-differentiation-new-dyadic}]

Let  $f \in L^1(\mathbb{R}^n; \C)$.  The claim of the theorem is equivalent to the assertion that
		\begin{align*}
			\C\left( \left\{ x\in \mathbb{R}^n : \limsup_{Q' \to x} \frac{1}{\C(Q')} \int_{Q'} \left|f - f(x) \right| \dc >0 \right\}\right)=0.
		\end{align*}
		
For $k \in \mathbb{N}$, define
		\begin{align*}
			A_k\vcentcolon= \left\{ x\in \mathbb{R}^n : \ \frac{1}{k} < \limsup_{Q' \to x} \frac{1}{\C(Q')} \int_{Q'} \left|f - f(x) \right| \dc \right\},
		\end{align*}
		and observe that
		\begin{align*}
			\left\{ x\in \mathbb{R}^n :  \limsup_{Q' \to x} \frac{1}{\C(Q')} \int_{Q'} \left|f - f(x) \right| \dc>0 \right\} = \bigcup_{k=1}^\infty A_k.
		\end{align*}
		Therefore, by countable subadditivity of $\C$, it suffices to show that $\C(A_k)=0$ for each $k \in \mathbb{N}$.

		Since $f \in L^1(\mathbb{R}^n; \C)$, by Proposition 3.2 of \cite{PonceSpector} (see Remark \ref{density_lp} in Section \ref{sec:preliminaries}), there exists a sequence of functions $\{f_l\}_l \subset C_b(\mathbb{R}^n) \cap L^1(\mathbb{R}^n;\C)$ such that 
		\begin{align*}
			\int_{\mathbb{R}^n} |f - f_l| \dc \leq \frac{1}{l}.
		\end{align*}
		For any $Q' \in \Dq$ with $\C(Q') > 0$, using triangle inequality, we get
		\begin{align*}
			&  \frac{1}{\C(Q')} \int_{Q'} \left|f(\cdot) - f(x) \right| \dc \\
			&\leq \frac{1}{\C(Q')} \int_{Q'} \{ |f(\cdot) - f_l(\cdot)| + |f_l(\cdot) - f_l(x)|+ |f_l(x) - f(x)|\} \dc \\ & \leq \frac{1}{\C(Q')}\int_{Q'}  |f(\cdot) - f_l(\cdot)| \dc + \frac{1}{\C(Q')}\int_{Q'} |f_l(\cdot) - f_l(x)| \dc+ |f_l(x) - f(x)|.
		\end{align*}
		Since $f_l$ is continuous for each $l$, we have 
		\begin{align*}
			\limsup_{Q' \to x}\frac{1}{\C(Q')} \int_{Q'} |f_l(\cdot) - f_l(x)| \dc = 0. 
		\end{align*}
		It follows that for every $\epsilon> 0$, we have the set inclusion 
		\begin{align*}
			A_k &\subset \left \{ x\in \mathbb{R}^n : \  M^{d}_{\C} (|f - f_l|)(x) > \frac{1}{2k}   \right\} \cup \left\{ x\in \mathbb{R}^n : \  |f_l(x) - f(x)| > \frac{1}{2k}  \right\}\\& : = A^{1,l}_k \cup A^{2,l}_k.
		\end{align*}
		By monotonicity and subadditivity, we have
		\begin{align*}
			\C(A_k) \leq  \C(A^{1,l}_k) + \C(A^{2,l}_k). 
		\end{align*}
		Finally, an application of Theorem \ref{max-new} (i) to $A^{1,l}_k$ and Chebychev's inequality to $A^{2,l}_k$, along with the density of $C_b(\mathbb{R}^n)$, yields that 
		\begin{align*}
			\limsup_{l \to \infty} \C(A^{1,l}_k)=0, \quad \text{ and }\quad 	\limsup_{l \to \infty} \C(A^{2,l}_k)=0.
		\end{align*}
		This completes the proof. 
	\end{proof}

	\section{Maximal Estimates, Differentiation Theory, and Calder\'on-Zygmund Decomposition for Translation Invariant Hausdorff Contents}\label{Max-Diff-CZ}	
	
	In this section, we prove Theorems \ref{weak-type-cont-maximal}, \ref{Lebesgue-diff-cont}, \ref{cz}.  We begin by establishing several lemmas that will be useful in the proofs.
	
	\begin{lemma}\label{content_doubling_lemma}
		Suppose $\varphi : [0,\infty] \to [0,\infty]$ is a monotone increasing function such that $\lim_{t \to 0^+} \varphi(t)=0$ and let $\cont$ be the dyadic content given in \eqref{dyadiccontent} defined in terms of $\varphi$.  Then $\cont$ is doubling and $\cont(\tau+E) = \cont(E)$ for all translations $\tau$ which preserve the lattice $\mathcal{D}(Q)$ and for all $E\subset \mathbb{R}^n$.
	\end{lemma}
	
	\begin{proof}
		First, for $E \subset \mathbb{R}^n$, we will prove that $\cont(\tau+E) = \cont(E)$ for all translations $\tau$ that preserve the lattice $\mathcal{D}(Q)$. 
		
		If 
		\begin{align*}
			E \subset \bigcup_{i} Q_i
		\end{align*}
		for $Q_i \in \mathcal{D}(Q)$, then
		\begin{align*}
			\tau + E \subset \bigcup_{i} \tau+Q_i,
		\end{align*}
		with $\tau+Q_i \in \mathcal{D}(Q)$. Since $\ell(Q_i)=\ell(\tau+Q_i)$ for each $i$ and $\tau$, therefore by definition of $\cont$, we get the desired equality.
		
		We next show that $\cont$ is doubling.  Let $B(x,r) \subset \mathbb{R}^n$ be given.  Then there exists a largest cube $Q'  \in \mathcal{D}(Q)$ such that $Q' \subset B(x,r)$ and translations on the lattice $\{\tau_j\}_{j=1}^M$ such that
		\begin{align*}
			B(x,2r) \subset \bigcup_{j=1}^M \tau_j + Q',
		\end{align*}
		where $M$ is a uniform dimension dependent constant.
		
		Then by monotonicity, subadditivity, and the fact that $\cont(Q')=\cont(\tau_j+Q')$, we obtain
		\begin{align*}
			\cont(B(x,2r)) &\leq \sum_{j=1}^M \cont(\tau_j + Q')\\
			&= M \cont(Q') \\
			&\leq M \cont(B(x,r)),
		\end{align*}
		which is the desired result.
	\end{proof}

	\begin{lemma}\label{eqivalencemaixmalleema}
		Suppose $\varphi : [0,\infty] \to [0,\infty]$ is a monotone increasing function such that $\lim_{t \to 0^+} \varphi(t)=0$ and let $\cont$ be the dyadic content given in \eqref{dyadiccontent} defined in terms of $\varphi$.  Then there holds
		\begin{align*}
			\cont ( \bigcup \{ 3Q': Q' \in \mathcal{D} (Q) \text{ and } Q' \subseteq E\} ) \leq 3^n \cont(E)
		\end{align*}
		for every $E \subset \mathbb{R}^n$.
	\end{lemma}
	
	\begin{proof}
		Let $E \subset \mathbb{R}^n$.  Then by the properties of the lattice we have
		\begin{align*}
			\bigcup \{ 3Q': Q' \in \mathcal{D} (Q) \text{ and } Q' \subseteq E\} = \bigcup \bigcup_{i=1}^{3^n} \{ \tau_i+Q': Q' \in \mathcal{D} (Q) \text{ and } Q' \subseteq E\},
		\end{align*}
		for a family of translations $\{\tau_i\}_{i=1}^{3^n}$.  Thus, finite subadditivity, monotonicity of $\cont$ and Lemma \ref{content_doubling_lemma} gives
		\begin{align*}
			& \cont ( \bigcup \{ 3Q': Q' \in \mathcal{D} (Q) \text{ and } Q' \subseteq E\} ) \\
			&\leq \sum_{i=1}^{3^n} \cont (\bigcup \{ \tau_i+Q': Q' \in \mathcal{D} (Q) \text{ and } Q' \subseteq E\}) \\
			&= 3^n\cont (\bigcup \{ Q': Q' \in \mathcal{D} (Q) \text{ and } Q' \subseteq E\})\\
			&\leq 3^n \cont(E),
		\end{align*}
		which is the desired result.
		
	\end{proof}

	\begin{lemma}\label{equivoftwomaximal}
		Suppose $\varphi : [0,\infty] \to [0,\infty]$ is a monotone increasing function such that $\lim_{t \to 0^+} \varphi(t)=0$, 
		and let $\cont$ be the dyadic content associated to $\varphi$ given in \eqref{dyadiccontent}.
		There exist constants $C > 0$ and $c > 0$ such that 
		\begin{align*}
			\cont\big(\{x \in \mathbb{R}^n : \mathcal{M}_{\cont}f(x) > t\}\big) \leq C \cont\big(\{x \in \mathbb{R}^n : M^d_{\cont}f(x) > ct\}\big),
		\end{align*}
		for every function $f: \mathbb{R}^n \to \mathbb{R}$, where $M^d_{\cont}$ denotes the dyadic maximal operator associated with $\cont$.
	\end{lemma}
	
	\begin{proof}
		Since $\cont$ satisfies the doubling property, we can assume without loss of generality that $\cont(B) > 0$ for any ball $B \subset \mathbb{R}^n$. Consider the uncentered maximal operator associated with $\cont$, defined as
		\begin{align*}
			\tilde{\mathcal{M}}_{\cont}f(x) \coloneqq \sup_{x \in B} \frac{1}{\cont(B)} \int_B |f| \dvv,
		\end{align*}
		where the supremum is taken over all balls $B$ containing $x$. Since $\tilde{\mathcal{M}}_{\cont}f(x) \geq \mathcal{M}_{\cont}f(x)$, it suffices to show that
		\begin{align*}
			\cont\big(\{x \in \mathbb{R}^n : \tilde{\mathcal{M}}_{\cont}f(x) > t\}\big) \leq C \cont\big(\{x \in \mathbb{R}^n : M^d_{\cont}f(x) > ct\}\big).
		\end{align*}
		
		Define the sets
		\begin{align*}
	E_t = \{x \in \mathbb{R}^n : \tilde{\mathcal{M}}_{\cont}f(x) > t\}, \quad \text{ and } \quad L_t = \{x \in \mathbb{R}^n : M^d_{\cont}f(x) > t\}.	    
		\end{align*}
		We claim that for every $x \in E_t$, there exist a universal constant $c>0$ depending on dimension and the doubling constant from Lemma \ref{content_doubling_lemma}, a ball $B^x \subseteq E_t$ containing $x$ and a dyadic cube $Q^x \in \mathcal{D}(Q)$ such that:
		\begin{enumerate}
			\item $Q^x \subseteq L_{ct}$,
			\item $x \in B^x \subseteq 3Q^x$,
		\end{enumerate}
		where $3Q^x$ is the cube with the same center as $Q^x$ but with three times its side length.
		
		To establish this claim, note that for $x \in E_t$, by the definition of $\tilde{\mathcal{M}}_{\cont}$, there exists a ball $B^x$ containing $x$ such that
		\begin{align*}
			\frac{1}{\cont(B^x)} \int_{B^x} |f| \dvv > t.
		\end{align*}
		Let $r_x$ denote the radius of $B^x$. By the properties of the dyadic grid $\mathcal{D}(Q)$, there exist at most $2^n$ disjoint dyadic cubes $\{Q_k^x\}_k$ such that:
		\begin{enumerate}
			\item $\frac{\ell(Q_k^x)}{2} < 2r_x \leq \ell(Q_k^x)$,
			\item $B^x \cap Q_k^x \neq \emptyset $ ,
			\item $B^x \subseteq \bigcup_k Q_k^x $.
		\end{enumerate}
		It follows that
		\begin{align*}
			\sum_k \int_{B^x \cap Q_k^x} |f| \dvv \geq \int_{B^x} |f| \dvv > \cont(B^x)\,t.
		\end{align*}
		Hence, there exists some $Q^x \in \{Q_k^x\}_k$ such that
		\begin{align*}
			\int_{B^x \cap Q^x} |f|  \dvv \geq \frac{1}{2^n} \cont(B^x)t.
		\end{align*}
		Using the doubling property of $\cont$ and properties $(1)$ and $(2)$ of the dyadic cube $Q^{x}$, there exists a dimension dependent constant $C' > 0$ such that
		\begin{align*}
			\frac{\cont(B^x)}{\cont(Q^x)} \geq C',
		\end{align*}
		which implies
		\begin{align*}
			\int_{Q^x} |f| \dvv \geq \frac{C'}{2^n} \cont(Q^x)t.
		\end{align*}
		Thus,
		\begin{align*}
			\frac{1}{\cont(Q^x)} \int_{Q^x} |f| \dvv > ct,
		\end{align*}
		where $c = \frac{C'}{2^n} > 0$. Consequently, $Q^x \subseteq L_{ct}$, and since $\frac{\ell(Q^x)}{2} < 2r_x \leq \ell(Q^x)$ and $B^x\cap Q^x \neq \emptyset$, it follows that $B^x \subseteq 3Q^x$, completing the claim.
		
		Using this construction, we have
		\begin{align*}
			\bigcup_{x \in E_t} B^x \subseteq \bigcup_{x \in E_t} 3Q^x \subseteq \bigcup\{3Q' : Q' \in \mathcal{D}(Q), Q' \subseteq L_{ct}\}.	
		\end{align*}
		
		By monotonicity of $\cont$, we obtain
		\begin{align}\label{com-f}
			\cont(E_t) \leq \cont\big(\bigcup_{x \in E_t} B^x\big) \leq \cont\big(\bigcup\{3Q' : Q' \in \mathcal{D}(Q), Q' \subseteq L_{ct}\}\big).	
		\end{align}
		Lemma \ref{eqivalencemaixmalleema} asserts the bound
		\begin{align*}
			\cont\big(\bigcup\{3Q' : Q' \in \mathcal{D}(Q), Q' \subseteq L_{ct}\}\big) \leq C \cont(L_{ct}),
		\end{align*}
		which in combination with \eqref{com-f} implies
		\begin{align*}
			\cont(E_t) \leq C \cont(L_{ct}).
		\end{align*}
		
		This completes the proof of the Lemma.
	\end{proof}
	
	\begin{proof}[Proof of Theorem \ref{weak-type-cont-maximal}]
		We first prove $(i)$.  By Proposition \ref{deco-cube-int}, $\cont$ satisfies (P).  Theorem \ref{max-new} $(i)$ asserts the estimate 
		\begin{align}\label{dyadic-cont-max-weak}
			\cont \left(\left\{ x\in \rn \, :\, M^d_{\cont} f(x) >t\right\}\right) \leq \frac{C}{t} \int_{\mathbb{R}^n} |f| \dvv,
		\end{align}
		for some constant $C>0$. The inequality \eqref{dyadic-cont-max-weak} and Lemma \ref{equivoftwomaximal} together imply
		\begin{align*}
			\cont \left(\left\{ x\in \rn \, :\, \mathcal{M}_{\cont} f(x) >t\right\}\right) \leq & C \cont \left(\left\{ x\in \rn \, :\, M^d_{\cont} f(x) >ct\right\}\right) \\
			\leq & \frac{C_1}{t} \int_{\mathbb{R}^n} |f|\; \dvv,
		\end{align*}
		which completes the proof of part $(i)$.
		
		We next argue $(ii)$.   The definition of the content maximal function implies that
		\begin{align*}
			\|\mathcal{M}_{\cont} f\|_{L^\infty(\cont)} \leq \|f\|_{L^\infty(\cont)}.
		\end{align*}
		This inequality, along with the inequality established in $(i)$, are sufficient to invoke \cite[Lemma 2.5]{ChenOoiSpector} for the quasi-sublinear operator $T=\mathcal{M}_{\cont}$.  This allows us to conclude that for $1<p<\infty$, there exists a constant $C^{\prime}>0$ such that 
		\begin{align*}
			\int_{\rn} (\mathcal{M}_{\cont} f)^p  \dvv \leq C^{\prime} \int_{\mathbb{R}^n} |f|^p \dvv.
		\end{align*}
		This completes the proof of $(ii)$.
	\end{proof}
	
	We next give the
	\begin{proof}[Proof of Theorem \ref{Lebesgue-diff-cont}]
		Let  $f \in L^1(\mathbb{R}^n; \cont)$.  The claim of the theorem is equivalent to the assertion that
		\begin{align*}
			\cont\left( \left\{ x\in \mathbb{R}^n :  \limsup_{r\to 0^+}\frac{1}{\cont(B(x,r))} \int_{B(x,r)} \left|f - f(x) \right| \dvv >0 \right\}\right)=0.
		\end{align*}
		
		Define
		\begin{align*}
			A_k\vcentcolon= \left\{ x\in \mathbb{R}^n : \ \frac{1}{k} < \limsup_{r\to 0^+}\frac{1}{\cont(B(x,r))} \int_{B(x,r)} \left|f - f(x) \right| \dvv \right\},
		\end{align*}
		and observe that
		\begin{align*}
			\left\{ x\in \mathbb{R}^n :  \limsup_{r\to 0^+}\frac{1}{\cont(B(x,r))} \int_{B(x,r)} \left|f - f(x) \right| \dvv >0 \right\} = \bigcup_{k=1}^\infty A_k.
		\end{align*}
		Therefore, by countable subadditivity of $\cont$, it suffices to show that $\cont(A_k)=0$ for each $k \in \mathbb{N}$.  
		
		The assumption that $f \in L^1(\mathbb{R}^n; \cont)$ allows us to utilize property $(3)$ of Theorem \ref{functionspace} (established in \cite[Proposition 3.2]{PonceSpector}) to find a sequence of functions $\{f_l\}_l \subset C_b(\mathbb{R}^n) \cap L^1(\mathbb{R}^n;\cont)$ such that 
		\begin{align}\label{density}
			\int_{\mathbb{R}^n} |f - f_l| \, \dvv \leq \frac{1}{l},
		\end{align}
		
		For any $r>0$ and $B(x,r)$ such that $\cont(B(x,r)) > 0$, using triangle inequality, we write 
		\begin{align*}
			&  \frac{1}{\cont(B(x,r))} \int_{B(x,r)} \left|f(\cdot) - f(x) \right| \dvv \\&\leq \frac{1}{\cont(B(x,r))} \int_{B(x,r)} \{ |f(\cdot) - f_l(\cdot)| + |f_l(\cdot) - f_l(x)|+ |f_l(x) - f(x)|\} \dvv \\ & \leq \frac{1}{\cont(B(x,r))}\int_{B(x,r)}  |f(\cdot) - f_l(\cdot)| \dvv \\&+ \frac{1}{\cont(B(x,r))}\int_{B(x,r)} |f_l(\cdot) - f_l(x)| \dvv+ |f_l(x) - f(x)|.
		\end{align*}
		Since $f_l$ is continuous for each $l$, we have 
		\begin{align*}
			\limsup_{r \to 0^+}\frac{1}{\cont(B(x,r))} \int_{B(x,r)} |f_l(\cdot) - f_l(x)| \dvv = 0. 
		\end{align*}
		It follows that for every $k  \in \mathbb{N}$, we have the set inclusion 
		\begin{align*}
			A_k &\subset \left \{ x\in \mathbb{R}^n : \  \mathcal{M}_{\cont} (|f - f_l|)(x) > \frac{1}{2k}   \right\} \cup \left\{ x\in \mathbb{R}^n : \  |f_l(x) - f(x)| > \frac{1}{2k}  \right\}\\& : = A^{1,l}_k \cup A^{2,l}_k.
		\end{align*}
		Therefore, monotonicity and subadditivity imply
		\begin{align*}
			\cont(A_k) \leq  \cont(A^{1,l}_k) + \cont(A^{2,l}_k). 
		\end{align*}
		Finally, an application of Theorem \ref{weak-type-cont-maximal} (i) to $A^{1,l}_k$ and Chebychev's inequality to $A^{2,l}_k$, along with \eqref{density}, yield that 
		\begin{align*}
			\limsup_{l \to \infty} \cont(A^{1,l}_k)=0, \quad \text{ and }\quad \limsup_{l \to \infty} \cont(A^{2,l}_k)=0.
		\end{align*}
		This completes the proof. 
	\end{proof}

	\begin{proof}[Proof of Theorem \ref{cz}]
		Consider the collection of dyadic cubes $\{ Q_i \} \subseteq \mathcal{D}(Q)$ that are contained in $Q^\prime$ and satisfy
		\begin{align*}
			\Lambda< \frac{1}{\cont(Q_i)} \int_{Q_i} |f| \dvv.
		\end{align*}
		Let $\{ Q_k \}$ be the maximal subcollection of $\{ Q_i \}$. Then by monotonicity of the outer capacity and the doubling property of $\cont$ established in Lemma \ref{content_doubling_lemma} and Remark \ref{doubling_consequence}, we have that for every $Q_k$, its parent cube $P_k$ satisfies
		\begin{align*}
			\Lambda \geq \frac{1}{\cont(P_k)} \int_{P_k} |f| \dvv \geq \frac{1}{M_0 \cont(Q_k)} \int_{Q_k} |f| \dvv,
		\end{align*}
		for some positive constant $M_{0}$ depending on the dimension and the doubling constant as in Lemma \ref{content_doubling_lemma}.
		This implies
		\begin{align*}
			\frac{1}{\cont(Q_k)} \int_{Q_k} |f| \dvv \leq M_0 \Lambda,
		\end{align*}
		which proves the first claimed property. 
		
		We next prove the second property.  If $ x \in Q' \setminus \bigcup_k Q_k $, then this means every dyadic cube $Q''$ contained in $ Q' $ that contains $ x$ satisfies
		\begin{align*}
			\frac{1}{\cont(Q'')} \int_{Q''} |f| \, \dvv \leq \Lambda.	
		\end{align*}
		As a result, Theorem \ref{Lebesgue-differentiation-new-dyadic} with $\C=\cont$, implies
		\begin{align*}
			|f(x)| \leq \Lambda
		\end{align*}
		for $\cont$-quasi every $ x \in Q' \setminus \bigcup_k Q_k $, which proves the claim.  This completes the proof of Theorem \ref{cz}.
		

	\end{proof}

	\section{A John-Nirenberg Inequality for Translation Invariant Hausdorff Contents}\label{JN-Inequality}
	In this section, we prove a John-Nirenberg inequality for $\cont$ i.e., Theorem \ref{jn_content}.  We begin with an exponential decay lemma established in \cite[Lemma 4.1]{ChenSpector}.
	\begin{lemma}\label{exponentialestimate}
		Suppose that for some $C',c'>1$, $F: (0,\infty) \to [0,\infty)$ satisfies
		\begin{align} \label{exp_decay}
			F(t) \leq C' \frac{F(t-c' s)}{s}
		\end{align}
		for all $c'^{-1} t > s \geq 1$.  If
		\begin{align*}
			F(t) \leq \frac{1}{t}
		\end{align*}
		for every $t>0$, then there exist constants $c,C>0$ such that
		\begin{align*}
			F(t) \leq  C \exp\{ -c t\}
		\end{align*}
		for every $t\geq c'$. Moreover, the constants $C$ and $c$ are explicit and are given by 
		\begin{align*}
			C=\exp\left(\frac{1}{C'e}+1\right), \quad \text{ and } \quad
			c=\frac{1}{C'c' e}.
		\end{align*}  
	\end{lemma}

Following the program of John and Nirenberg \cite{JohnNirenberg}, and the argument of \cite[Lemma 4.2]{ChenSpector}, we first prove a version of the inequality for cubes with possibly infinite length. 
	\begin{lemma}\label{aux_lemma}
		Let $Q_0\in\Dq$. Then there exist constants $c,C>0$ such that
		\begin{align}\label{jninequality_content}
			& \cont\left(\{x\in Q^\prime  :|f(x)-c_{Q^\prime}|>t\}\right)\\ &\leq \frac{C}{\|f\|_{BMO^{\cont}(Q_0)}}  \left(\int_{Q^\prime}  |f-c_{Q^\prime}|\dvv \right) \exp\left(\frac{- c t}{\|f\|_{BMO^{\cont}(Q_0)}}\right)\notag
		\end{align}
		for all $f \in BMO^{\cont}(Q_0)$, all finite subcubes $Q^\prime\in \mathcal{D}(Q_0)$, and $t\geq c' \|f\|_{BMO^{\cont}(Q_0)}$, where
		\begin{align}
			c_{Q^\prime}= \argmin_{c \in \mathbb{R}} \frac{1}{\cont(Q^\prime)}  \int_{Q^\prime}  |f-c| \dvv \label{c_Q}
		\end{align}
		and $c'=2+2 M_0$, where $M_0$ is the same positive constant appearing in Theorem \ref{cz}.
	\end{lemma}
	
	\begin{proof}
		It is easy to see that the inequality (\ref{jninequality_content}) holds when $\Vert f \Vert_{BMO^{\cont}(Q_0)} =0$. Thus, we may assume without loss of generality, $c_{Q^\prime}=0$ and $\|f\|_{BMO^{\cont}(Q_0)}=1$, by replacing $f$ with $\frac{f-c_{Q^\prime}}{\|f\|_{BMO^{\cont}(Q_0)}}$.  Let $F(t)$ be the smallest function for which there holds
		\begin{align}\label{jninequality_content_F}
			\cont\left(\{x\in Q^\prime:|f(x)|>t\}\right) \leq F(t)  \int_{Q^\prime}  |f|\dvv,
		\end{align}
		for all $t>0$, $f \in BMO^{\cont} (Q_0)$ with $\Vert f\Vert_{BMO^{\cont} (Q_0)}=1$ and finite subcubes $Q^\prime\in \D( Q_0)$. 
		
		Note that using Chebychev’s inequality and \eqref{jninequality_content_F}, we have
		\begin{align*}
			F(t) \leq \frac{1}{t},
		\end{align*}
for all $t>0$. We next show that
		\begin{align}\label{Ft-esti}
			F(t) \leq C' \frac{F(t-c' s)}{s}
		\end{align}
		for all $c'^{-1}t > s \geq 1$, where $c'=2+2 M_0$. Once we prove the estimate \eqref{Ft-esti}, the desired inequality is followed by an application of Lemma \ref{exponentialestimate}.  
		
		Since $c_{Q^\prime}=0$ and $\|f\|_{BMO^{\cont}(Q_0)}=1$, we have
		\begin{align} \label{upperbound}
			\frac{1}{\cont( Q^\prime)} \int_{Q^\prime}  |f| \dvv \leq 1.
		\end{align}
		
		If $c'^{-1}t > s \geq 1$, the choice
		$\Lambda = s$ is an admissible height for an application of Theorem \ref{cz} to $|f|$ to obtain a countable collection of non-overlapping dyadic cubes $\{ Q_j\} \subset Q^\prime$  such that 
		\begin{enumerate}
			\item $ s< \frac{1}{\cont(Q_j)} \int_{Q_j} |f| \dvv \leq M_0 s$,
			\item $|f(x)|\leq  s$ for $\cont$ -quasi everywhere $x\in Q^\prime \setminus \cup_j Q_j$.
		\end{enumerate}
		An application of Proposition \ref{decom-cubes} to this family $\{ Q_j\}$ yields a subfamily $\{ Q_{j_k}\}$ and non-overlapping ancestors $\{ \tilde{Q}_k\}$ which satisfy
		\begin{enumerate}
			\item \begin{align*}
				\bigcup_{j} Q_j \subset \bigcup_{k} Q_{j_k} \cup \bigcup_{k} \tilde{Q}_k,
			\end{align*}
			\item
			\begin{align*}
				\sum_{Q_{j_k} \subset P} \cont(Q_{j_k}) \leq 2\cont(P), \text{ for each dyadic cube } P\in \Dq ,
			\end{align*}
			\item
			\begin{align*}
				\cont(\tilde{Q}_k) \leq \sum_{Q_{j_i} \subset \tilde{Q}_k}\cont(Q_{j_i}), \text{ for each } \tilde{Q}_k .
			\end{align*}
		\end{enumerate}
		In what follows, we assume that we avoid redundancy in this covering, i.e., when $Q_{j_k} \subset \tilde{Q}_m$ for some $k,m$, we do not use the cube $Q_{j_k}$.  
		
		Since $t > c' s > s$, we have
		\begin{align*}
			\{x\in Q^\prime:|f(x)|>t\} \subset \bigcup_{j} Q_j \cup N \subset \bigcup_{k, Q_{j_k} \not \subset \tilde{Q}_{m}} Q_{j_k} \cup \bigcup_{k} \tilde{Q}_k \cup N,
		\end{align*}
		where $N$ is a null set with respect to $\cont$.  Thus
		\begin{align*}
			\{x\in Q^\prime:|f(x)|>t\} \subset  \bigcup_{k, Q_{j_k} \not\subset \tilde{Q}_{m}}  &\{x\in Q_{j_k}:|f(x)|>t\} \\&\cup \bigcup_{k}  \{x\in \tilde{Q}_k:|f(x)|>t\} \cup N. 
		\end{align*}
		Therefore, by using countable subadditivity of the dyadic content $\cont$, we have
		\begin{align}\label{subadditiveconsequence}
			\nonumber
			\cont\left(\{x\in Q^\prime:|f(x)|>t\}\right) &\leq \sum_{k, Q_{j_k} \not \subset \tilde{Q}_{m}} \cont\left(\{x\in Q_{j_k}:|f(x)|>t\}\right) \\
			&\;\;+ \sum_k \cont\left(\{x\in \tilde{Q}_k :|f(x)|>t\}\right).
		\end{align}
We next show that if $R=Q_{j_k},\tilde{Q}_k$, then
		\begin{align*}
			|c_{R}| \leq c' s,
		\end{align*}
		where $c_{R}$ is defined in \eqref{c_Q}.
		
		It is easy to see that $0 = c_{R} \leq c' s$ when $\cont(R) = 0$. When $\cont(R) > 0$, we have by the definition of Choquet integral and using $s\geq 1$,
		\begin{align*}
			|c_{R}| &= \frac{1}{\cont(R)} \int_{R} |c_{R}| \dvv \\
			&\leq 2 \left( \frac{1}{\cont(R)} \int_{R} |f-c_{R}|\dvv + \frac{1}{\cont(R)} \int_{R} |f|\dvv \right) \\
			&\leq 2 + 2 M_0  s \\
			&\leq (2+2 M_0) s\\
			&= c' s.
		\end{align*}
		The above estimate implies
		\begin{align}\label{esti-u}
			|f| \leq |f-c_{R}| + |c_{R}| 
			\leq  |f-c_{R}| + c' s.
		\end{align} 
		Therefore, using estimates \eqref{subadditiveconsequence} and \eqref{esti-u}, we write
		\begin{align*}
			\cont\left(\{x\in Q^\prime:|f(x)|>t\}\right) &\leq \sum_{k, Q_{j_k} \not \subset \tilde{Q}_{m}} \cont\left(\{x\in Q_{j_k}:|f(x)-c_{Q_{j_k}}|>t-c' s\}\right) \\
			&\;\;+ \sum_k \cont\left(\{x\in \tilde{Q}_{k}:|f(x)-c_{\tilde{Q}_{k}}|>t-c' s\}\right).
		\end{align*}
The fact $\|f-c_{Q'}\|_{BMO^{\cont}(Q_0)} =  \|f\|_{BMO^{\cont}(Q_0)} =1$ and the definition of $F$ gives
		\begin{align}\label{JN-prelim}
			& \cont\left(\{x\in Q^\prime:|f(x)|>t\}\right)\\
			\nonumber &\leq \sum_{k, Q_{j_k} \not \subset \tilde{Q}_{m}} F(t-c' s) \int_{Q_{j_k}} |f-c_{Q_{j_k}}|\dvv
			\;\;+ \sum_k F(t-c' s) \int_{\tilde{Q}_{k}} |f-c_{\tilde{Q}_{k}}|\dvv\\
			\nonumber &\leq \sum_{k, Q_{j_k} \not \subset \tilde{Q}_{m}}  F(t-c' s) \cont(Q_{j_k}) + \sum_k  F(t-c' s) \cont(\tilde{Q}_{k})\\
			\nonumber	&\leq  F(t-c' s) \sum_k  \cont(Q_{j_k}),
		\end{align}
		where in the last line we use the fact that there is no redundancy in cubes in the collections.  As the subcollection of cubes $\{Q_{j_k}\}$ comes from the Calder\'on-Zygmund decomposition, each $Q_{j_k}$ satisfies
		\begin{align}\label{CZ-esti}
			\cont(Q_{j_k}) < \frac{1}{ s} \int_{Q_{j_k}} |f|\dvv,
		\end{align}
		and $\cont(Q_{j_k}) > 0 $ for each of $Q_{j_k}$. 
		
		Using \eqref{JN-prelim}, \eqref{CZ-esti} and \eqref{pack-int-hyp-2}, we deduce
		\begin{align*}
			\cont\left(\{x\in Q^\prime:|f(x)|>t\}\right) &\leq F(t-c' s) \sum_k  \frac{1}{ s} \int_{Q_{j_k}} |f| \dvv \\
			&\leq F(t-c' s) \frac{C'}{ s} \int_{\bigcup_k Q_{j_k}} |f|\dvv\\
			&\leq F(t-c' s) \frac{C'}{s} \int_{Q^\prime} |f| \dvv.
		\end{align*}
		
		When $\int_{Q^\prime} |f| \dvv =0$, by Chebychev’s inequality, \eqref{jninequality_content} follows immediately. Hence, we may assume  $\int_{Q^\prime} |f| \dv >0$. Moreover, as $F(t)$ is the smallest function satisfying \eqref{jninequality_content_F}, so we have
		\begin{align*}
			\cont\left(\{x\in Q^\prime:|f(x)|>t\}\right) > \frac{F(t)}{2}  \int_{Q^\prime}  |f|\dvv,
		\end{align*}
		and thus
		\begin{align*}
			\frac{F(t)}{2} < \frac{\cont\left(\{x\in Q^\prime:|f(x)|>t\}\right)}{\int_{Q^\prime} |f|\dvv} \leq C' \frac{F(t- c' s)}{s}.
		\end{align*}
		This completes the proof of \eqref{Ft-esti} and therefore the lemma.
	\end{proof}
	
	Now, we present the main result of this section so called John-Nirenberg inequality for the dyadic content $\cont$.
	
	\begin{proof}[Proof of Theorem \ref{jn_content}]
		The Chebychev's inequality settles the case $\Vert f \Vert_{BMO^{\cont}(Q_0)} =0$. Without loss of generality, we may assume $\Vert f \Vert_{BMO^{\cont}(Q_0)} >0$.
		From the proof of Lemma \ref{jninequality_content}, we have
		\begin{align*}
			&\cont\left(\{x\in Q^\prime:|f(x)-c_{Q^{\prime}}|>t\}\right) \\ &\leq C \exp\left(\frac{-ct}{\|f\|_{BMO^{\cont}(Q_0)}}\right)  \frac{1}{\|f\|_{BMO^{\cont}(Q_0)}} \int_{Q\prime}  |f-c_{Q\prime}| \dvv \\
			&\leq C  \exp\left(\frac{-ct}{\|f\|_{BMO^{\cont}(Q_0)}}\right)  \cont(Q^\prime)
		\end{align*}
		for all $t \geq c' \|f\|_{BMO^{\cont}(Q_0)}$.  But if $t \in (0,c' \|f\|_{BMO^{\cont}(Q_0)})$, we have
		\begin{align*}
			&\cont\left(\{x\in Q^\prime:|f(x)-c_{Q^\prime}|>t\}\right)\\ &\leq  \cont(Q^\prime) \\
			&\leq \cont(Q^\prime)\exp\left(\frac{-t}{c' \|f\|_{BMO^{\cont}(Q_0)}}+1\right) \\
			&= \cont(Q^\prime)\exp(1) \exp\left(\frac{-t}{c' \|f\|_{BMO^{\cont}(Q_0)}}\right) \\
			&\leq  \cont(Q^\prime)\exp\left(\frac{1}{2C'e}+1\right) \exp\left(\frac{-t}{2C'c' e\|f\|_{BMO^{\cont}(Q_0)}}\right)\\
			&= C \cont(Q^\prime) \exp\left(\frac{-ct}{\|f\|_{BMO^{\cont}(Q_0)}}\right),
		\end{align*}
		where 
		\begin{align*}
			C=\exp\left(\frac{1}{2C'e}+1\right) \quad \text{ and } \quad 
			c=\frac{1}{2C'c' e}
		\end{align*}
		are the same constants as derived in Lemma \ref{exponentialestimate}. This completes the proof of Theorem~\ref{jn_content}.
	\end{proof}

	
	\section{Extension of the Theory to Outer Capacities}\label{capacities}
	The main result established in this section is Theorem \ref{com-cap}.  We also show how Theorem \ref{com-cap} and the preceding analysis easily implies Theorem \ref{maxi-esti-new}, \ref{Lebesgue-differentiation-new}, \ref{cz_C}, and \ref{jn_content_C}.
	
	%

	%

	We require the following result that a capacity $\C$ and its induced dyadic content $\contC$ coincide on dyadic cubes.
	\begin{proposition}\label{equality}
		Suppose $\C$ is an outer capacity.	Then for all $Q'\in\Dq$, we have $\C(Q')=\contC(Q')$.
	\end{proposition}
	\begin{remark}
		In fact, this holds for all outer measures, without assuming the outer regularity condition (iv).
	\end{remark}
	\begin{proof}
		It is easy to notice that $Q'$ is a cover of itself. Therefore, we have
		\begin{align*}
			\contC(Q')\leq \C(Q').
		\end{align*}
		Let $\{ Q_i\}_i \subseteq \mathcal{D}(Q)$ be a collection of dyadic cubes such that $Q'\subseteq \cup_{i} Q_i$. Then using the monotonicity and countable subadditivity of $\C$ we can write
		\begin{align*}
			\C(Q')\leq \C\left(\cup_{i} Q_i\right)\leq \sum_{i}\C(Q_i).
		\end{align*} 
		Now, taking the infimum over all such cover we deduce
		\begin{align*}
			\C(Q')\leq \contC(Q'),
		\end{align*}
		which gives the other side implications.
	\end{proof}

	We have the necessary preliminaries in place to complete the proof of Theorem~\ref{com-cap}.
	\begin{proof}[Proof of Theorem \ref{com-cap}]
		Suppose, $\cup_{i} Q_i$, for $Q_i\in\Dq$ is a cover of $E$. Then by properties of the capacity $\C$, we have
		\begin{align*}
			\C(E)\leq \sum_{i}\C(Q_i).
		\end{align*}
		Taking the infimum of all such covers we obtain
		\begin{align*}
			\C(E) \leq \contC(E).
		\end{align*}
		We next prove the corresponding lower bound.  To this end, let $U$ be any open set containing $E$.  By Whitney's decomposition theorem for open sets, there exist non-overlapping cubes $Q_j\in \Dq$, such that $U=\bigcup_{j}Q_j$, see e.g.~ \cite[J.1 on p.~09]{Grafakos1class} (note that while the claim states that the cubes are closed, in the construction, they are half-open and the conclusion is valid with either choice).
		Thus we have 
		\begin{align*}
			E\subset U= \bigcup_{j}Q_j.
		\end{align*}
		Now, by Proposition \ref{decom-cubes} for the family $\{Q_j\}_j$, there exists a subfamily $\left\{Q_{j_{k}}\right\}_k$ and a family of nonoverlapping ancestors $\{\tilde{Q}_{k}\}_k$ such that there holds
		\begin{align}\label{com-pack-1}
			\bigcup_{j} Q_{j} \subset \bigcup_{k} Q_{j_{k}} \cup \bigcup_{k} \tilde{Q}_{k},
		\end{align}
		\begin{align}\label{com-pack-2}
			\sum_{Q_{j_{k}} \subset Q^\prime} \contC(Q_{j_{k}}) \leq 2 \contC(Q^\prime), \text { for each } Q^\prime\in\Dq,
		\end{align}
		and
		\begin{align}\label{com-pack-3}
			\contC(\tilde{Q}_{k}) \leq \sum_{Q_{j_{i}} \subset \tilde{Q}_{k}} \contC(Q_{j_{i}}).
		\end{align}
		Now by applying Proposition~\ref{equality} on each dyadic cubes, \eqref{com-pack-2} turns out that 
        \begin{align*}
			\sum_{Q_{j_{k}} \subset Q^\prime} \C(Q_{j_{k}}) \leq 2 \C(Q^\prime), \text { for each } Q^\prime\in\Dq.
		\end{align*}
        Hence using the above and by the Packing Assumption (P), for the subfamily $\left\{Q_{j_{k}}\right\}_k$, and $f=\boldsymbol{1}_{\cup_{k} Q_{j_{k}}}$, we have
		\begin{equation}\label{com-pack-4}
			\sum_{k}  \C(Q_{j_{k}}) \leq 2\, \C(\cup_{k} Q_{j_k}).
		\end{equation}
		In the coming part in many places, by Proposition \ref{equality}, we will use the fact that for any $Q^\prime\in \Dq$, there holds $\C(Q^\prime)=\cont(Q^\prime)$. Now, from monotonicity and countable subadditivity property of $\C$, we have
		\begin{align*}
			\C(U)&\geq \C\left(\cup_{k} Q_{j_{k}}\right) \\&\overset{\eqref{com-pack-4}}{\geq} \frac{1}{2} \sum_{k}  \C(Q_{j_{k}})\\&= \frac{1}{2} \sum_{k}  \contC(Q_{j_{k}})\\&= \frac{1}{4} \sum_{k}  \contC(Q_{j_{k}})+\frac{1}{4} \sum_{k}  \contC(Q_{j_{k}}) \\& \overset{\eqref{com-pack-3}}{\geq } \frac{1}{4} \sum_{k}  \contC(Q_{j_{k}})+\frac{1}{4} \sum_{k}  \contC(\tilde{Q}_{k})\\&\geq \frac{1}{4}\contC\left(\bigcup_{k} Q_{j_{k}} \cup \bigcup_{k} \tilde{Q}_{k}\right)\\&\overset{\eqref{com-pack-1}}{\geq}\frac{1}{4} \contC\left(\bigcup_{j} Q_{j}\right)\\&\geq \frac{1}{4} \contC(E).
		\end{align*} 
		Now, taking infimum over such open sets and using outer regularity of $\C$, we get
		\begin{align*}
			\frac{1}{4}\contC(E) \leq \C(E).
		\end{align*}
		This completes the proof of the desired lower bound and therefore Theorem \ref{com-cap}.
	\end{proof}
	
	\subsection{Proofs of Theorems \ref{maxi-esti-new}, \ref{Lebesgue-differentiation-new}, \ref{cz_C} and \ref{jn_content_C}}
	Because of Theorem \ref{com-cap}, the arguments of the proofs of Theorems \ref{maxi-esti-new}, \ref{Lebesgue-differentiation-new}, \ref{cz_C} and \ref{jn_content_C} follow along the lines of Theorems \ref{weak-type-cont-maximal}, \ref{Lebesgue-diff-cont}, \ref{cz} and \ref{jn_content}.  However, there are some differences in the proofs based on the fact that, unlike $\cont$, $\contC$ is not translation invariant.  In particular, we first establish analogues of Lemmas \ref{content_doubling_lemma} and \ref{eqivalencemaixmalleema} for $\contC$.

	\begin{lemma}\label{doubling-contC}
		Suppose $\C$ is an outer capacity which satisfies (P) and \eqref{global_doubling}.  Then $\contC$ is doubling.
	\end{lemma}
	
	\begin{proof}
		Let $B(x,r) \subset \mathbb{R}^n$.  The assumptions of the lemma are sufficient to apply Theorem \ref{com-cap}, which in combination with the doubling assumption on $\C$ gives
		\begin{align*}
			\contC(B(x,2r)) \leq 4 \C(B(x,2r)) \leq 4 D \C(B(x,r)) \leq 4 D \contC(B(x,r)),
		\end{align*}
		which is the desired conclusion.
	\end{proof}
	
	\begin{lemma}\label{content_3cube_estimate}
		Suppose $\C$ is an outer capacity which satisfies (P) and \eqref{global_doubling}.  Then there exists a constant $C$ such that 
		\begin{align*}
			\contC ( \bigcup \{ 3Q': Q' \in \mathcal{D} (Q) \text{ and } Q' \subseteq E\} ) \leq C \contC(E)
		\end{align*}
		for every $E \subset \mathbb{R}^n$.
	\end{lemma}
	
	\begin{proof}
		Fix $\epsilon>0$.   By the definition of $\contC$, we may find a family of dyadic cubes $\{Q_i\}_{i}$ such that
		\begin{align*}
			E \subset \bigcup_i Q_i, \quad \text{ and }\quad 			\sum_{i=1}^\infty \C(Q_i) &\leq \contC(E) +\epsilon.
		\end{align*}
		By monotonicity and countable subadditivity of $\C$, we may assume that $\{Q_i\}_{i}$ is maximally disjoint.  Therefore, we have
		\begin{align}\label{inclusion}
			\bigcup \{ 3Q': Q' \in \mathcal{D} (Q) \text{ and } Q' \subseteq E\} \subset \bigcup_{i} 3Q_i.
		\end{align}
		Hence, the above estimate \eqref{inclusion} together with monotonicity, countable subadditivity, doubling property (Lemma \ref{doubling-contC}) of $\contC$ and Proposition \ref{equality}  imply
		\begin{align*}
			\contC(\bigcup \{ 3Q': Q' &\in \mathcal{D} (Q) \text{ and } Q' \subseteq E\}) \leq \contC(\bigcup_{i} 3Q_i)\\
			&\leq \sum_{i=1}^\infty \contC(3Q_i)\\
			&\leq A\sum_{i=1}^\infty \contC(Q_i)\\
			&= A\sum_{i=1}^\infty  \C(Q_i)\\
			&\leq A\contC(E) +A\epsilon.
		\end{align*}
		The claim follows by sending $\epsilon \to 0$. 
	\end{proof}
	
	\begin{lemma}\label{equivoftwomaximal-1}
		Suppose $\C$ is an outer capacity which satisfies (P) and \eqref{global_doubling}. Then there exist constants $C > 0$ and $c > 0$ such that 
		\begin{align*}
			\contC\big(\{x \in \mathbb{R}^n : \mathcal{M}_{\contC}f(x) > t\}\big) \leq C \contC\big(\{x \in \mathbb{R}^n : M^d_{\contC}f(x) > ct\}\big),
		\end{align*}
		for every function $f: \mathbb{R}^n \to \mathbb{R}$, where $M^d_{\contC}$ denotes the dyadic maximal operator associated with $\contC$.
	\end{lemma}
	\begin{proof}
		The proof of Lemma \ref{equivoftwomaximal-1} follows the same argument as that in Lemma \ref{equivoftwomaximal}, where an analogous result was established for $\cont$. The main components of the proof of Lemma \ref{equivoftwomaximal} were Lemmas \ref{content_doubling_lemma} and \ref{eqivalencemaixmalleema}, which one replaces here with Lemmas \ref{doubling-contC} and \ref{content_3cube_estimate}, while all other aspects of the proof of Lemma \ref{equivoftwomaximal} are quite general and are satisfied by outer measures. 
Therefore, we omit the details.
	\end{proof}

	\begin{proof}[Proof of Theorem \ref{maxi-esti-new}]
		By Theorem \ref{com-cap}, since $\contC \sim \C$, it is enough to prove the Theorem \ref{maxi-esti-new} for $\contC$. 
		First, we prove $(i)$. Proposition \ref{deco-cube-int} shows that $\contC$ satisfies (P) and therefore, using  Theorem \ref{max-new}$(i)$, we get 
		\begin{align}\label{dyadic-contC-max-weak}
			\contC \left(\left\{ x\in \rn \, :\, M^d_{\contC} f(x) >t\right\}\right) \leq \frac{C}{t} \int_{\mathbb{R}^n} |f|\dcontC
		\end{align}
		for some constant $C>0$. 
		By Lemma \ref{equivoftwomaximal-1} and the above inequality \eqref{dyadic-contC-max-weak}, we have
		\begin{align*}
			\contC \left(\left\{ x\in \rn \, :\, \mathcal{M}_{\contC} f(x) >t\right\}\right) \leq & C \contC \left(\left\{ x\in \rn \, :\, M^d_{\contC} f(x) >ct\right\}\right) \\
			\leq & \frac{C_1}{t} \int_{\mathbb{R}^n} |f|\dcontC.
		\end{align*}
		This completes the proof of part $(i)$.
		
		The part $(ii)$ of the Theorem \ref{maxi-esti-new} is follows from the estimate $\|\mathcal{M}_{\contC}f\|_{L^{\infty}(\contC)} \leq \|f\|_{L^{\infty}(\contC)}$ along with the part $(i)$ i.e., weak type $(1,1)$ estimate and the Lemma 2.5 of \cite{ChenOoiSpector}. 
		
		This completes the proof of the theorem.
	\end{proof} 
	
	\begin{proof}[Proof of Theorems \ref{Lebesgue-differentiation-new}]
		Again, by Theorem \ref{com-cap}, it suffices to prove Theorem \ref{Lebesgue-differentiation-new} for the corresponding content $\contC$. The proof then follows the usual argument of the Lebesgue differentiation theorem presented in Theorems \ref{Lebesgue-differentiation-new-dyadic} and \ref{Lebesgue-diff-cont}, whose details we omit for brevity.
	\end{proof}
	
	\begin{proof}[Proofs of Theorems \ref{cz_C} and \ref{jn_content_C}]
		We make a final appeal to Theorem \ref{com-cap} so that it suffices to prove Theorems \ref{cz_C} and \ref{jn_content_C} for the corresponding content $\contC$ only. According to Lemma \ref{doubling-contC}, we know that $\contC$ is doubling.  Thus, as $\contC$ satisfies the Lebesgue differentiation theorem, as stated in Theorem \ref{Lebesgue-differentiation-new-dyadic}, the remainder of the proof for Theorem \ref{cz_C} follows the argument of Theorem \ref{cz}.      
		
		Finally, the proof of Theorem \ref{jn_content_C} follows almost in the line of the proof of Theorem \ref{jn_content}.  In particular, one observes that the main ingredients used in proving Theorem \ref{jn_content} were Proposition \ref{decom-cubes}, Proposition \ref{deco-cube-int}, and the Calder\'on-Zygmund decomposition (Theorem \ref{cz}).  As Propositions \ref{decom-cubes} and \ref{deco-cube-int} are shown to be valid for general monotone increasing set functions, they are applicable to $\contC$ while in place of Theorem \ref{cz} one utilizes Theorem \ref{cz_C}.  The remainder of the proof is unchanged, and we again omit the details for brevity.
		
		This completes the proof of Theorem \ref{cz_C} and \ref{jn_content_C}.
		
\end{proof}

	\section*{Acknowledgments} 
	R.~Basak is supported by the National Science and Technology Council of Taiwan under research grant numbers 113-2811-M-003-007/113-2811-M-003-039. Y.-W.~B.~Chen is supported by the National Science and Technology Council of Taiwan under research grant number 113-2811-M-002-027. P.~Roychowdhury is supported by the National Theoretical Science Research Center Operational Plan under project number 112L104040.   D.~Spector is supported by the National Science and Technology Council of Taiwan under research grant numbers 110-2115-M-003-020-MY3/113-2115-M-003-017-MY3 and the Taiwan Ministry of Education under the Yushan Fellow Program.

\end{document}